\documentclass[12pt,reqno]{amsart}

\usepackage{amssymb, amsmath, amsthm, amsfonts}
\usepackage{mathabx}
\usepackage{thmtools}
\usepackage{mathtools}
\usepackage[centering, margin=1in]{geometry}
\usepackage{enumitem}
\usepackage[space]{cite}
\usepackage{verbatim}
\usepackage[T1]{fontenc}
\usepackage{graphicx}

\usepackage[colorlinks, linkcolor=black, citecolor=black, hypertexnames=false]{hyperref}
\usepackage[capitalize, nameinlink]{cleveref}

\usepackage[nobysame, non-compressed-cites]{amsrefs}

\AtBeginDocument{\def\MR#1{}}

\mathtoolsset{showonlyrefs} 

\declaretheorem[name=Theorem, numberwithin=section]{theorem}

\declaretheorem[name=Lemma, sibling=theorem]{lemma}
\declaretheorem[name=Proposition, sibling=theorem]{prop}
\declaretheorem[name=Corollary, sibling=theorem]{corollary}
\declaretheorem[name=Definition, sibling=theorem]{definition}

\declaretheorem[name=Remark, sibling=theorem]{remark}

\declaretheorem[name=Example, sibling=theorem]{example}

\declaretheorem[name=Claim, numbered=no]{claim*}

\crefname{section}{Section}{Sections}
\Crefname{section}{Section}{Sections}

\crefname{appendix}{Appendix}{Appendices}
\Crefname{appendix}{Appendix}{Appendices}

\numberwithin{equation}{section}

\usepackage{mathrsfs}
\usepackage{amssymb}
\usepackage{dsfont}
\usepackage{verbatim}
\usepackage{url}
\usepackage{mathtools}
\usepackage{etoolbox}
\usepackage{leftidx}

\def\Z{\mathbb {Z}}
\def\Q{\mathbb {Q}}
\def\R{\mathbb {R}}
\def\C{\mathbb {C}}

\DeclareMathOperator{\cond}{cond}

\def\Hom{\qopname\relax o{Hom}}

\DeclareMathOperator{\ord}{ord}

\newcommand{\xdashrightarrow}[2][]{\ext@arrow 0359\rightarrowfill@@{#1}{#2}}

\DeclareRobustCommand
  \rddots{\mathinner{\mkern1mu\raise\p@
    \vbox{\kern7\p@\hbox{.}}\mkern2mu
    \raise4\p@\hbox{.}\mkern2mu\raise7\p@\hbox{.}\mkern1mu}}

\DeclareMathAlphabet{\mathcal}{OMS}{cmsy}{m}{n}

\newcommand\cC{\mathcal{C}}

\newcommand\cE{\mathcal{E}}

\newcommand\cN{\mathcal{N}}
\newcommand\cO{\mathcal{O}}

\newcommand\cS{\mathcal{S}}

\newcommand\sL{\mathscr{L}}

\newcommand\frakh{\mathfrak{h}}

\newcommand\frakm{\mathfrak{m}}
\newcommand\frakn{\mathfrak{n}}

\newcommand\frakr{\mathfrak{r}}

\newcommand\diff{\mathop{}\!\mathrm{d}}

\newcommand\dcross{\diff^{\kern-2pt\raisebox{1pt}{$\times$}}\kern-8pt}

\newcommand\blfootnote{\xdef\@thefnmark{}\@footnotetext}

\DeclareFontFamily{U}{wncy}{}
\DeclareFontShape{U}{wncy}{m}{n}{<->wncyr10}{}
\DeclareSymbolFont{mcy}{U}{wncy}{m}{n}
\DeclareMathSymbol{\Sha}{\mathord}{mcy}{"58}

\renewcommand{\pmod}[1]{\, (\mathrm{mod}\ #1)}

\newcommand{\p}{\mathfrak{p}}

\newcommand{\normt}[1]{\left\lVert#1\right\rVert_2}

\renewcommand{\a}{\mathfrak a}
\renewcommand{\b}{\mathfrak b}
\renewcommand{\c}{\mathfrak c}
\renewcommand{\d}{\mathfrak d}
\newcommand{\e}{\mathfrak e}
\newcommand{\f}{\mathfrak f}
\newcommand{\g}{\mathfrak g}
\renewcommand{\r}{\mathfrak r}
\newcommand{\s}{\mathfrak s}
\renewcommand{\q}{\mathfrak q}
\renewcommand{\u}{\mathfrak u}
\newcommand{\h}{\mathfrak h}

\newcommand{\tchi}{\tilde{\chi}}

\def\sumstar{\operatornamewithlimits{\sum\nolimits^\star}}
\def\sumflat{\operatornamewithlimits{\sum\nolimits^\flat}}

\newcommand*{\bbone}{\text{\usefont{U}{bbold}{m}{n}1}}

\begin{document}

\author{Alexandre de Faveri}
\address{EPFL SB MATH, Station 10, 1015 Lausanne, Switzerland}
\email{\url{alexandre.defaveri@epfl.ch}}

\title{Optimal large sieve for fixed order characters}

\begin{abstract}
    Let $n \geq 3$. We improve the large sieve inequality for $n$-th order Hecke characters over any number field containing the $n$-th roots of unity. Our bound is expected to be optimal in all ranges of parameters. Several applications are deduced, including a tight upper bound for the second moment of cubic characters ordered by conductor, asymptotics with lower order terms for the cubic second moment over squarefree moduli, and improved moment and zero density estimates for every $n \geq 3$.

\end{abstract}

\maketitle


\section{Introduction}

The influential quadratic large sieve inequality of Heath-Brown \cite{HB95} states that
\begin{equation}\label{eq:quadratic_large_sieve}
    \sum_{\substack{0 < a \leq A \\ (a, 2)=1}} \mu^2(a) \Bigg| \sum_{\substack{0 < b \leq B \\ (b, 2)=1}} \mu^2(b) \lambda_b \Big(\frac{a}{b}\Big)_2\Bigg|^2 \ll_\varepsilon (AB)^\varepsilon (A+B) \normt{\lambda}^2
\end{equation}
for any complex-valued sequence $\{\lambda_b\}_{b\in \Z}$, $A, B \geq 1$, and $\varepsilon > 0$. This fundamental bound has had striking applications, for instance to moments and non-vanishing of $L$-functions \cites{Sou00, CS02, SY10, RS15, DW21, Li24} and to arithmetic statistics \cites{FK07, FK10, BHB20}. 

In this paper, we study versions of \eqref{eq:quadratic_large_sieve} for higher order characters. Those were pioneered by Heath-Brown \cite{HB00} for cubic characters, extended to arbitrary order by Blomer, Goldmakher, and Louvel \cite{BGL14}, and also used for arithmetic applications \cites{Kly20, DDDS24, CDD26, DIPP26}. Our main result is a large sieve inequality which should be optimal in logarithmic scale for characters of all orders.

\subsection{Results} 

Let $n \geq 2$ be an integer and $K$ be a number field containing the $n$-th roots of unity, with ring of integers $\cO$ and absolute norm $\cN(\cdot)$. Let $S$ be a finite set of finite places of $K$, containing the places dividing $n$ and such that the ring of $S$-integers has class number $1$. Denote by $I(S)$ the set of integral ideals of $\cO$ coprime with $S$. 

Let $\{\chi_\a\}_{\a \in I(S)}$ be a family of $n$-th order Hecke characters\footnote{We fix the data used to construct this family (including $n, K, S$), so implied constants may depend on it throughout the paper.} introduced in \cites{FF04, FHL03} and defined (for $n\geq 3$) in \cref{sec:background}. Consider
\begin{equation*}
    \Xi_n(A, B) := \sup_{\normt{\lambda} = 1} \sumstar_{\substack{\cN(\a) \leq A}} \Big| \sumstar_{\substack{\cN(\b) \leq B}} \lambda_\b \chi_\a(\b) \Big|^2
\end{equation*}
and 
\begin{equation}\label{eq:Theta_n_def}
    \Theta_n(A, B) := \sup_{\normt{\lambda} = 1} \sumflat_{\substack{\cN(\a) \leq A}} \Big| \sumflat_{\substack{\cN(\b) \leq B}} \lambda_\b \chi_\a(\b) \Big|^2.
\end{equation}
Here and henceforth, $\sumstar$ denotes summation over squarefree ideals of $I(S)$, while $\sumflat$ denotes summation over $n$-th power free\footnote{That is, ideals not divisible by the $n$-th power of a prime ideal of $\cO$.} ideals of $I(S)$. Observe that
\begin{equation}\label{eq:trivial_inequality}
    \Xi_n(A, B) \leq \Theta_n(A, B).
\end{equation}

Heath-Brown's quadratic large sieve \cite{HB95} and its generalization to number fields by Goldmakher and Louvel \cite{GL13} (see also \cite{Ono09}) give
\begin{equation}\label{eq:quadratic_bound}
    \Xi_2(A, B) \ll_\varepsilon (AB)^\varepsilon (A+B).
\end{equation}
This bound is optimal, apart from the term $(AB)^\varepsilon$. We refer to \cites{Li24, SS24, Liu26} for refinements.

For $n \geq 3$, the best known upper bound is
\begin{equation}\label{eq:BGL_bound}
    \Xi_n(A, B) \ll_\varepsilon (AB)^\varepsilon (A+B + (AB)^{2/3}).
\end{equation}
This is due to Heath-Brown \cite{HB00} for $n = 3$ with $K = \Q(\zeta_3)$, and to Blomer, Goldmakher, and Louvel \cite{BGL14} in full generality (see also \cites{BY10, GZ12}). It was a folklore belief that the ``fully orthogonal'' bound \eqref{eq:quadratic_bound} should also hold for $\Xi_n(
A, B)$ when $n \geq 3$. However, this was disproved for $n=3$, assuming the generalized Riemann hypothesis (GRH), by Dunn and Radziwi{\l}{\l} \cite{DR24}. More precisely, they showed that
\begin{equation}\label{eq:cubic_lower_bound}
    \Xi_3^{\text{HB}}(A, B) \gg_\varepsilon (AB)^{-\varepsilon} (A+B + (AB)^{2/3}),
\end{equation}
where $\Xi_3^{\text{HB}}$ denotes the operator norm for the specific cubic family (with $K=\Q(\zeta_3)$) considered by Heath-Brown in \cite{HB00}, which we describe in \cref{ex:cubic}. 

Recent work of the author with Dunn and Hoffstein \cite{DDH26} established the lower bound \eqref{eq:cubic_lower_bound} unconditionally, and also proved an unconditional lower bound for $n=4$ of the form
\begin{equation}\label{eq:quartic_lower_bound}
    \Xi_4^{\text{DDH}}(A, B) \gg_\varepsilon (AB)^{-\varepsilon} (A + B + A^{3/4}B^{1/2} + A^{1/2}B^{3/4}),
\end{equation}
where $\Xi_4^{\text{DDH}}$ denotes the operator norm for a natural quartic family (with $K=\Q(i)$) described in \cref{ex:quartic}. These lower bounds are obtained from biases coming from higher order Gauss sums, which are detected via Rankin--Selberg theory \cite{Zag81} applied to metaplectic forms \cite{DDH26}*{Section 5}. The same strategy should extend to general $n \geq 3$, and the resulting heuristic lower bounds led us to conjecture \cite{DDH26}*{(1.8)} that
\begin{equation}\label{eq:conjecture}
    \Xi_n(A, B) = (AB)^{o(1)} (A + B + A^{1-1/n} B^{2/n} + A^{2/n} B^{1-1/n}) \quad \text{ as } \quad \min(A, B) \to \infty.
\end{equation}

Our main result proves the upper bound part of this conjecture, in a stronger form.

\begin{theorem}[Large sieve upper bound]\label{thm:main}
    Let $n \geq 3$. Then for every $A, B \geq 1$ and $\varepsilon > 0$ we have
    \begin{equation*}
        \Theta_n(A, B) \ll_\varepsilon (AB)^\varepsilon (A + B + A^{1-1/n}B^{2/n} + A^{2/n}B^{1-1/n}).
    \end{equation*}
\end{theorem}

In view of \eqref{eq:trivial_inequality} and \eqref{eq:conjecture}, we expect the bound of \cref{thm:main} to be optimal in logarithmic scale. The matching lower bounds \eqref{eq:cubic_lower_bound} and \eqref{eq:quartic_lower_bound} establish this optimality for the corresponding families with $n=3$ and $n=4$.

\subsection{Some applications}

For each $\a \in I(S)$, let $L(s, \chi_\a)$ be the Hecke $L$-function of $\chi_\a$ and $\cond(\chi_\a)$ denote its conductor. Two families of primitive $n$-th order Hecke characters $\chi_\a$ are typically considered in the literature: the family of squarefree $\a \in I(S)$, and the family of $n$-th power free $\a \in I(S)$, ordered by $\cN(\a)$ in both cases. 

For the second family, it would be more natural to order by the norm $\cN(\cond(\chi_\a))$ of the conductors. This ordering has so far been mostly neglected, in part because it is difficult to analyze. For instance, the large sieve (even with our improvements) does not give a tight upper bound for the second moment of $L(s, \chi_\a)$ with this ordering. 

In our first application, we use some structural features of the case $n=3$ to establish a tight upper bound for the cubic second moment ordered by conductor. 

\begin{theorem}[Second moment for cubic characters ordered by conductor]\label{thm:cubic_second_moment}
    Let $n=3$. Let $\psi$ be a unitary Hecke character over $K$. Then for every $A, B \geq 1$ and $\varepsilon > 0$ we have
    \begin{equation*}
        \sum_{\substack{\a \in I(S) \\ \cN(\a) \leq A}} \sum_{\substack{\b \in I(S) \\ \cN(\b) \leq B}} \mu^2(\a\b) \cdot |L(\tfrac{1}{2}, \psi\chi_{\a\b^2})|^2 \ll_{\psi, \varepsilon} (AB)^{1+\varepsilon}.
    \end{equation*}
    In particular, for $X \geq 1$,
    \begin{equation*}
        \sumflat_{\substack{\cN(\cond(\chi_\g)) \leq X}} |L(\tfrac{1}{2}, \chi_\g)|^2 \ll_\varepsilon X^{1+\varepsilon}.
    \end{equation*}
\end{theorem}

The proof of \cref{thm:cubic_second_moment} does not require \cref{thm:main}, but uses similar ideas. It would be interesting to obtain analogues of this moment bound for $n\geq 4$.

Further motivation for proving \cref{thm:cubic_second_moment} comes from a recent observation of Diaconu, Ion, Pa{\char179}ol, and Popa \cite{DIPP26}*{Remark 1.1.3} that such a bound would improve the error term in the second moment over the squarefree cubic family ordered by norm. This leads to our next result, which for the first time detects (when $n=3$) the secondary main term conjectured in 2004 by Diaconu \cite{Dia04}*{Conjecture 4.5} for the second moment of the squarefree family\footnote{For the $n$-th power free family, secondary main terms were detected in \cite{DIPP26} when $n=3$ and $n=4$.}.

Let $L^S(s, \chi_\a)$ denote $L(s, \chi_\a)$ with the local factors at primes in $S$ removed.

\begin{theorem}[Cubic second moment with secondary term]\label{thm:cubic_secondary_term}
    Let $n = 3$. Assume\footnote{This technical condition is inherited from \cite{DIPP26}, ultimately coming from \cite{BB06}.} that $S$ contains all primes ramified in $K/\Q$. Fix a Schwartz function $W:\R \to \R_{\geq 0}$ supported on $(1, 2)$. Then there exist linear polynomials $P_W$ and $Q_W$ such that for any $X \geq 1$ and $ \varepsilon >0$ we have\footnote{Note that our result does not rule out the case $Q_W \equiv 0$.}
    \begin{equation*}
        \sum_{\a \in I(S)} \mu^2(\a) \cdot |L^S(\tfrac{1}{2}, \chi_\a)|^2 \cdot W\Big( \frac{\cN(\a)}{X}\Big) = X P_W(\log{X}) + X^{5/6} Q_W(\log{X}) + O_{\varepsilon}(X^{3/4+\varepsilon}).
    \end{equation*}
\end{theorem}

Asymptotics with power-saving error terms for the cubic second moment in \cref{thm:cubic_secondary_term} were first obtained by David, the author, Dunn, and Stucky \cite{DDDS24}, with error term $O_\varepsilon(X^{5/6+\varepsilon})$. Previously, power-saving asymptotics were only available in the quadratic case ($n=2$), after the breakthrough work of Soundararajan \cite{Sou00} in 2000.

Interestingly, the same error term $O_\varepsilon(X^{5/6+\varepsilon})$ for the cubic second moment as in \cite{DDDS24} was also recovered in \cite{DIPP26} using different methods, based on multiple Dirichlet series.

Due to structural reasons, we do not yet have moment bounds as strong as in the cubic case for general $n$. However, we also improve the second moment error term in the squarefree family for every $n\geq 4$.

\begin{theorem}[Improved second moment error term]\label{thm:general_second_moment}
    Let $n \geq 4$. Assume that the base field $K$ contains the $2n$-th roots of unity\footnote{The assumption that $K$ contains the $2n$-th roots of unity (instead of $n$-th roots as before) is a technical artifact of \cite{DIPP26}. The authors observe \cite{DIPP26}*{Remark 1.4.1} that it should be possible to remove this condition.} and that $S$ contains all primes ramified in $K/\Q$. Fix a Schwartz function $W:\R \to \R_{\geq 0}$ supported on $(1, 2)$. Then there exists a linear polynomial $P_{n, W}$ such that for any $X \geq 1$ and $ \varepsilon >0$ we have
    \begin{equation*}
        \sum_{\a \in I(S)} \mu^2(\a) \cdot |L^S(\tfrac{1}{2}, \chi_\a)|^2 \cdot W\Big( \frac{\cN(\a)}{X}\Big) = X P_{n, W}(\log{X}) + O_{\varepsilon}(X^{\sigma_n+\varepsilon}),
    \end{equation*}
    where $\sigma_4 = \frac{56}{65}$ and $\sigma_n = 1-\frac{n-1}{2(n+1)(n-2)}$ for $n \geq 5$.
\end{theorem}

Asymptotics with power-saving error terms for the second moment of \cref{thm:general_second_moment} were first obtained by Castillo, the author, and Dunn \cite{CDD26} for $n=4$, and by Diaconu, Ion, Pa{\char179}ol, and Popa \cite{DIPP26} for all $n$. In \cite{CDD26}, we were motivated by applications to ranks of elliptic curves, and for simplicity gave an unspecified power-saving. In \cite{DIPP26}*{Theorem A}, the authors obtained explicit error terms for all $n$. Our \cref{thm:general_second_moment} improves these error terms for every $n$. 

We also improve the zero density estimates of \cite{BGL14}*{Corollary 1.6} for $n$-th order characters.

\begin{corollary}[Improved zero density estimate]\label{thm:zero_density}
    For $\tfrac{1}{2} < \sigma < 1$, $T > 1$ and $\a \in I(S)$, let $N(\sigma, T, \a)$ denote the number of zeros $\rho$ of $L(s, \chi_{\a})$ in the rectangle $\sigma \leq \Re(\rho) \leq 1$ and $|\Im(\rho)| \leq T$, counted with multiplicity. Then for any $N>1$ and $\varepsilon > 0$ we have
    \begin{equation*}
        \sumflat_{\cN(\a) \leq N} N(\sigma, T, \a) \ll_\varepsilon N^{1 - \delta_n(\sigma) + \varepsilon} \cdot T^{1 + \frac{d(1-\sigma)}{3-2\sigma} + \varepsilon},
    \end{equation*}
    where $d := [K:\Q]$ and
    \begin{equation*}
        \delta_n(\sigma) := \begin{cases}
            \frac{2\sigma-1}{(1-2\sigma)n+4} & \text{ for } \frac{1}{2}\leq \sigma \leq \frac{1}{2} + \frac{1}{n}, \\[1mm]
            \frac{(2\sigma-1)((2\sigma-1)n^2-(6\sigma-2)n+3)}{(2\sigma-1)(3-2\sigma)n^2-4\sigma n+2} & \text{ for } \frac{1}{2} + \frac{1}{n}\leq \sigma \leq 1 - \frac{1}{2n}, \\[1mm]
            \frac{8\sigma^2-10\sigma+3}{1-4(1-\sigma)^2n} & \text{ for } 1 - \frac{1}{2n} \leq \sigma \leq  1.
        \end{cases}  
    \end{equation*}
\end{corollary}

\begin{figure}[ht]
  \centering
  \includegraphics[width=.62\linewidth]
    {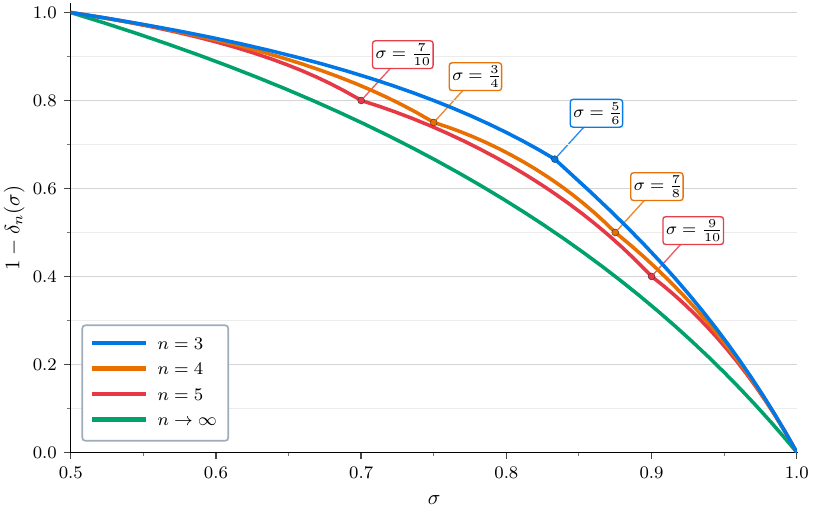}
  \caption{Plot of the zero density exponent $1-\delta_n(\sigma)$ in \cref{thm:zero_density} for $n=3, 4, 5$, and in the limit $n\to\infty$. Generated by ChatGPT 5.6 Sol.}
  \label{fig:zero-density-exponent}
\end{figure}

The bound of Blomer, Goldmakher, and Louvel \cite{BGL14} corresponds to the blue line in \cref{fig:zero-density-exponent} for every $n\geq 3$. Note also that \cref{thm:zero_density} applies to the larger $n$-th power free family.

We expect \cref{thm:main} to have several other applications, for instance to moments in other families and to non-vanishing for the corresponding $L$-functions. 

\subsection{Discussion of the main ideas}

\subsubsection{The large sieve inequality}

The main new idea in the proof of \cref{thm:main} is to bound the larger kernel $\Theta_n$, where both sums are over $n$-th power free ideals, instead of the usual kernel $\Xi_n$, where the sums are over squarefree ideals. 

To describe what is gained, let us give a sketch of the proof. The main maneuver is to complete the outer sum over $\a \sim A$ in \eqref{eq:Theta_n_def} to all integral ideals by positivity, open up the square, and apply Poisson summation in $\a$. One hopes to obtain roughly\footnote{In this sketch we repeatedly use the reciprocity relation $\chi_\a(\b) \approx \chi_\b(\a)$, ignoring local factors.}
\begin{equation}\label{eq:sketch_Poisson}
    \sumflat_{\a\sim A} \Big|\sumflat_{\b\sim B} \lambda_\b \chi_\b(\a) \Big|^2 \leq \sum_{\a\sim A} \Big|\sumflat_{\b\sim B} \lambda_\b \chi_\b(\a) \Big|^2 \approx \frac{A}{B} \sum_{\a\sim B^2/A} \Big|\sumflat_{\b\sim B} \lambda'_\b \chi_\b(\a) \Big|^2 
\end{equation}
after Poisson, since the resulting characters $\chi_{\b_1} \overline{\chi_{\b_2}}(\a)$ typically have conductor $\approx B^2$. If the sum over $\a$ can be reduced to a sum over $n$-th power free ideals, then this process could be iterated\footnote{One also needs to inflate one of the ranges when this process is self-dual (for instance if $A=B$), but this inflation is now standard, so we focus on \eqref{eq:sketch_Poisson}.}. 

Writing $\a = \d \e^{n}$ for $n$-th power free $\d$, and localizing $\e \sim E$, we crucially have $E \ll (B^2/A)^{1/n}$. Trivially bounding the sum over $\e$, the contribution of this range to \eqref{eq:sketch_Poisson} is
\begin{equation}\label{eq:sketch_nth_powers}
    \ll \frac{A}{B}\sum_{\e \sim E} \sumflat_{\d \sim B^2/AE^n} \Big|\sumflat_{\substack{\b\sim B \\ (\b, \e)=1}} \lambda'_{\b} \chi_\b(\d) \Big|^2  \ll \frac{EA}{B} \cdot  \Theta_n\Big( \frac{B^2}{AE^n}, B\Big) \cdot \normt{\lambda}^2.
\end{equation}
Note that $\Theta_n\big( \frac{B^2}{AE^n}, B\big) \gg B$, so the best bound we can expect from this argument is
\begin{equation*}
    \Theta_n(A, B) \ll \frac{EA}{B}\cdot B = EA \ll \Big(\frac{B^2}{A}\Big)^{1/n} A = A^{1-1/n}B^{2/n}.
\end{equation*}
This explains the shape of the bound of \cref{thm:main}. 

If instead we worked with the squarefree kernel $\Xi_n$, then decomposing $\a = \a_1 \a_2^2 \cdots \a_n^{n}$ for $\a_1, \a_2, \dots, \a_{n-1}$ squarefree and localizing $\a_i \sim A_i$, then the only clear move is to trivially bound all except the longest of the sums over the $\a_i$ with $1\leq i \leq n-1$. Even in the case $A_1 = A_2 = (B^2/A)^{1/3}$ and $A_3 = \cdots = A_n = 1$, the procedure above only shows that $\Xi_n(A, B)$ is at most
\begin{equation*}
    \frac{A}{B} \cdot \Big(\frac{B^2}{A}\Big)^{1/3} \cdot \Xi_n\Big(\Big(\frac{B^2}{A}\Big)^{1/3}, B\Big) \gg \frac{A}{B} \cdot \Big(\frac{B^2}{A}\Big)^{1/3} \cdot B = (AB)^{2/3}. 
\end{equation*}
This explains the previous bounds obtained in \cites{HB00, BGL14}.

\subsubsection{Coupling of root numbers}

Returning to our argument, one important obstacle arises in \eqref{eq:sketch_Poisson}. There we intended to apply Poisson summation. Denoting $C = \cN(\cond(\chi_{\b_1}\overline{\chi_{\b_2}}))$, this takes the form
\begin{equation}\label{eq:sketch_corrected_Poisson}
    \sum_{\a \sim A} \chi_{\b_1} \overline{\chi_{\b_2}}(\a) \approx \epsilon(\chi_{\b_1}\overline{\chi_{\b_2}}) \cdot \frac{A}{C^{1/2}} \sum_{\a \sim C/A} \overline{\chi_{\b_1}}\chi_{\b_2}(\a),
\end{equation}
where $\epsilon(\chi_{\b_1}\overline{\chi_{\b_2}})$ is the root number. For simplicity, assume we are in the generic case $(\b_1, \b_2)=1$. If $\b_1$ and $\b_2$ are squarefree, then $C \asymp \cN(\b_1\b_2) \asymp B^2$ and $\epsilon(\chi_{\b_1}\overline{\chi_{\b_2}}) \approx \epsilon(\chi_{\b_1}) \cdot \overline{\epsilon(\chi_{\b_2})}$. This splitting of the root number allows one to recombine the sums over $\b_1$ and $\b_2$ into a square, and is a feature of all previous works. However, if $\b_1$ and $\b_2$ are not squarefree, as in our case, then the root number does not split and \eqref{eq:sketch_Poisson} is too simplistic. 

Writing $\b_1^\star$ for the product of prime ideals dividing $\b_1$ (i.e.\ its squarefree kernel) and $\b_1^\sharp := \b_1 / \b_1^\star$, and similarly for $\b_2$, then up to irrelevant local factors we have 
\begin{equation*}
    \epsilon(\chi_{\b_1}\overline{\chi_{\b_2}}) \approx \epsilon(\chi_{\b_1}) \cdot \overline{\epsilon(\chi_{\b_2})} \cdot \chi_{\b_2^\star}(\b_1^\sharp) \cdot \overline{\chi_{\b_1^\star}(\b_2^\sharp)}.
\end{equation*}
Localizing $\b^\sharp_i \sim B_i$ and $\b_i^\star \sim B/B_i$, generically we have $\cN(\cond(\chi_{\b_1} \overline{\chi_{\b_2}})) \asymp \cN(\b_1^\star \b_2^\star) \asymp B^2/(B_1B_2)$. By \eqref{eq:sketch_corrected_Poisson}, the right side of \eqref{eq:sketch_Poisson} for these ranges should be
\begin{equation}\label{eq:sketch_pre_Cauchy}
    \approx \frac{A(B_1B_2)^{1/2}}{B} \sumflat_{\b_1^\sharp \sim B_1} \sumflat_{\b_2^\sharp \sim B_2} \sum_{\a \sim \frac{B^2}{AB_1B_2}} \cS_1 \cdot \overline{\cS_2}
\end{equation}
for 
\begin{equation}
    \cS_{1} \approx \sumstar_{\b_1^\star \sim B/B_1} \lambda_{\b_1^\sharp \b_1^\star} \epsilon(\chi_{\b_1^\sharp \b_1^\star}) \overline{\chi_{\b_1^\star}(\b_2^\sharp) \chi_{\b_1^\sharp \b_1^\star}(\a)}
\end{equation}
and analogously for $\cS_2$. We dropped some conditions (such as $\b_1^\sharp \b_1^\star$ being $n$-th power free) to simplify the exposition.

Now apply Cauchy--Schwarz to the sums over $\b_1^\sharp, \b_2^\sharp, \a$ to recover the desired shape as in \eqref{eq:sketch_Poisson}. At this point it is not enough to fix $\b_1^\sharp$ and $\b_2^\sharp$ and ignore the shape of the couplings $\b_1^\star \leftrightarrow \b_2^\sharp$ and $\b_2^\star \leftrightarrow \b_1^\sharp$, since if $B_1 = B_2 = B/A^{1/2}$ we would only get the insufficient bound
\begin{equation*}
    A^{1/2} \frac{B}{A^{1/2}} \cdot  \Theta_n(1, A^{1/2}) \Big( \sumflat_{\b_1^\sharp \sim \frac{B}{A^{1/2}}} \sumstar_{\b_1^\star \sim A^{1/2}} |\lambda_{\b_1^\sharp \b_1^\star}|^2 \Big)^{1/2} \Big( \sumflat_{\b_2^\sharp \sim \frac{B}{A^{1/2}}} \sumstar_{\b_2^\star \sim A^{1/2}} |\lambda_{\b_2^\sharp \b_2^\star}|^2 \Big)^{1/2} \gg A^{1/2}B \cdot \normt{\lambda}^2.
\end{equation*} 

\subsubsection{Using the coupling structure}

Note that for $\lambda'_\b := \lambda_\b \epsilon(\chi_\b)$ we have
\begin{equation*}
    \sumflat_{\b_1^\sharp \sim B_1} \sumflat_{\b_2^\sharp \sim B_2} \sum_{\a \sim \frac{B^2}{AB_1B_2}} |\cS_1|^2 \approx \sumflat_{\b_1^\sharp \sim B_1} \sumflat_{\b_2^\sharp \sim B_2} \sum_{\a \sim \frac{B^2}{AB_1B_2}} \Big|\sumstar_{\b_1^\star \sim B/B_1} \lambda'_{\b_1^\sharp \b_1^\star}\overline{\chi_{\b_1^\star}(\b_2^\sharp \a)}\Big|^2.
\end{equation*}
Combining $\b_2^\sharp \a$ into a single variable and applying a divisor bound, the display above is
\begin{equation}\label{eq:sketch_inside_Cauchy}
    \ll_\varepsilon (AB)^\varepsilon \sumflat_{\b_1^\sharp \sim B_1} \sum_{\d \sim \frac{B^2}{AB_1}} \Big|\sumstar_{\b_1^\star \sim B/B_1} \lambda'_{\b_1^\sharp \b_1^\star}\overline{\chi_{\b_1^\star}(\d)}\Big|^2.
\end{equation}

This is similar to \eqref{eq:sketch_Poisson}, except both sums are shorter by a factor of $B_1$. The iterated bound available for this kernel will be smaller than the one for \eqref{eq:sketch_Poisson} by at least a factor of $B_1$ (since terms in such a bound have sum of exponents $\geq 1$ by AM-GM). The extra sum over $\b_1^\sharp$ in \eqref{eq:sketch_inside_Cauchy} is irrelevant, as it only serves to gather $\normt{\lambda}^2$. Hence recalling that \eqref{eq:sketch_inside_Cauchy} is inside a square-root, we at least cancel the factor of $B_1^{1/2}$ in \eqref{eq:sketch_pre_Cauchy}. Thus the worst case is $B_1 = B_2 = 1$, in which case \eqref{eq:sketch_pre_Cauchy} becomes exactly the ideal bound in \eqref{eq:sketch_Poisson}.

\subsubsection{Cubic moments in conductor aspect}

We briefly comment on the proof of \cref{thm:cubic_second_moment}, which is quite simple and uses structural properties of $n = 3$. By the approximate functional equation, for $\psi = 1$ the theorem essentially reduces to a bound of the form
\begin{equation*}
    \sum_{\substack{\a \sim A}} \sum_{\substack{\b \sim B}} \Big| \sumstar_{\substack{\q \sim (AB)^{1/2}}}  \chi_{\q}(\a \b^2) \Big|^2 \ll_\varepsilon (AB)^{3/2+\varepsilon}.
\end{equation*}
The squarefree condition on $\q$ can be introduced for free due to the decay of the $L$-function coefficients $\cN(\q)^{-1/2}$, and we may assume $A \leq B$ by conjugating the inner sum.

Inspired by the proof of \cref{thm:main}, we open the square and apply Poisson summation to $\chi_{\q_1}\overline{\chi_{\q_2}}(\a\b^2)$ in the longer sum over $\b$. Since the conductor is typically of size $\cN(\q_1\q_2) \asymp AB$, the dual length will be $AB/B = A$. This does not improve the lengths when $A = B$, but it dramatically improves the shape of the sum. After Poisson the character is conjugated, hence we obtain 
\begin{equation*}
    \approx \frac{B}{(AB)^{1/2}} \sum_{\a\sim A} \sum_{\b \sim A} \Big| \sumstar_{\substack{\q \sim (AB)^{1/2}}}  \chi_{\q}(\a \b) \Big|^2.
\end{equation*}
At this point we can combine $\a\b$ into a single variable, which has norm $A^2$ (compared to $AB^2$ in the initial sum). Removing cubes as in \eqref{eq:sketch_nth_powers} and using \cref{thm:main} we get
\begin{equation*}
    \ll \sup_{1 \ll E \ll A^{2/3}} BE \cdot \Theta_3\Big(\frac{A^2}{E^3}, (AB)^{1/2}\Big) \ll_\varepsilon (AB)^\varepsilon B \big(A^2 + A^{2/3}(AB)^{1/2} + (A^2(AB)^{1/2})^{2/3} \big),
\end{equation*}
which is $\ll_\varepsilon (AB)^{3/2+\varepsilon}$ since $A\leq B$. For details, see the more general \cref{prop:cubic_conductor_bound}.

\subsection*{Acknowledgments}

The author thanks Alexander Dunn for many inspiring conversations about $n$-th order characters throughout the years, and for helpful comments on this work.

\subsection*{AI use}

All of the ideas and writing in this paper are due to the author. We used ChatGPT to create the plot in \cref{fig:zero-density-exponent}, to write Sage code for the two examples in \cref{sec:examples}, to compute the optimal exponents in the last step of the proofs of \cref{thm:general_second_moment} and \cref{thm:zero_density}, and to proofread. 


\section{Background}\label{sec:background}

\subsection{Number fields}

Let $n \geq 2$ be an integer and $K$ be a number field containing the $n$-th roots of unity, with ring of integers $\cO$ and absolute norm $\cN(\cdot)$. For each place $v$ of $K$, let $K_v$ denote the completion of $K$ at $v$. For $v$ finite, let $\p_v$ denote the corresponding prime ideal of $\cO$. For any integral ideal $\a \subseteq \cO$, let
\begin{equation*}
    \a^\star := \prod_{n \nmid \ord_v(\a)} \p_v
\end{equation*}
denote the squarefree kernel of the $n$-th power free part of $\a$.

\subsection{The \texorpdfstring{$n$}{}-th power residue symbol}

The $n$-th power residue symbol in $K$ is defined, for each prime ideal $\p \subset \cO$ with $(\p, n) = 1$ and $\alpha \in \cO$ with $(\alpha, \p) = 1$, as the unique $n$-th root of unity satisfying
\begin{equation}\label{eq:residue_symbol_prime}
    \Big(\frac{\alpha}{\p}\Big)_n \equiv \alpha^{\frac{\cN(\p)-1}{n}} \pmod{\p}.
\end{equation}
The symbol can be extended multiplicatively to any fractional ideal $\mathfrak{b}$ of $\cO$ with $(\mathfrak{b}, n\alpha) = 1$ by
\begin{equation}\label{eq:residue_symbol_def}
     \Big(\frac{\alpha}{\mathfrak{b}}\Big)_n  := \prod_{\substack{\p \text{ prime}}} \Big(\frac{\alpha}{\p}\Big)^{\nu_\p(\mathfrak{b})}_n.
\end{equation}
We also multiplicatively extend the symbol to $\alpha \in K^\times$ if $(\alpha, \b) = 1$, and set $\big(\frac{\alpha}{\b}\big)_n := 0$ if $(\alpha, \b) \neq 1$. See \cite{CF67}*{Exercise 1} for more details on the $n$-th power residue symbol.

\begin{remark}
    Note that \eqref{eq:residue_symbol_prime} and \eqref{eq:residue_symbol_def} define $\big(\frac{\alpha}{\b}\big)_n$ as an element of $K$, and to obtain a complex number we fix a field embedding $K \hookrightarrow \C$ once and for all, which will be implicit in the notation. A different choice of embedding leads to $\big(\frac{\alpha}{\b}\big)_n \mapsto \big(\frac{\alpha}{\b}\big)^r_n$ for some $r \in (\Z/n\Z)^\times$.
\end{remark}

\subsection{Extension to \texorpdfstring{$n$}{}-th order Hecke characters}\label{sec:n_th_order_construction}

The $n$-th power residue symbols can be extended to a family of Hecke characters of order (dividing) $n$. This extension was constructed by Fisher and Friedberg \cite{FF04} for $n=2$, and by Friedberg, Hoffstein, and Lieman \cite{FHL03} for general $n$. A detailed description is given in \cite{GL13}*{Section 2}, which we mostly follow. For the purposes of this paper we may assume that $n \geq 3$, which we do from now on to simplify some of the exposition below.

\subsubsection{The set $S$}

We fix once and for all a finite set $S$ of finite places of $K$, containing all divisors of $n$ and such that denoting $\widetilde{S} := S \sqcup S_\infty$, where $S_\infty$ is the set of all infinite places of $K$, then the ring $\cO_{\widetilde{S}}$ of $\widetilde{S}$-integers has class number one. Let $I(S)$ and $I^*(S)$ denote respectively the sets of integral and fractional ideals of $\cO$ coprime with $S$.

\subsubsection{The group $R_\c$}

Fix also an integral ideal $\c := \prod_{v \in S} \p_v^{e_v}$, for integers $e_v \geq 1$ chosen sufficiently large so that $\ord_{\p_v}(a-1) \geq e_v$ for $a \in K_v$ implies $a \in (K_v^\times)^n$. In particular, if $\p_v \nmid n$ then we may choose $e_v = 1$. Let $H_\c$ be the ray class group modulo $\c$, and let\footnote{The reader may prefer to think of $R_\c \simeq H_\c / H_\c^n$.} $R_\c := H_\c \otimes \Z/n\Z$. For $\a \in I^*(S)$, let $[\a] \in R_\c$ denote its class. By the choice of $\c$, note that Hecke characters of order dividing $n$ unramified outside $S$ factor through $R_\c$. Then class field theory and Kummer theory give $\Hom(R_\c, \mu_{n}) \simeq \cO_{\widetilde{S}}^\times / (\cO_{\widetilde{S}}^\times)^n$. The $S$-unit theorem and duality imply
\begin{equation}\label{eq:R_C_product}
    R_\c \simeq (\Z/n\Z)^{\frac{[K:\Q]}{2} + |S|}.
\end{equation}

\subsubsection{The set $\cE$}

Choose a generator for each $(\Z/n\Z)$-component of $R_\c$ in \eqref{eq:R_C_product}, and let $\cE_0$ be a complete set of ideals of $\cO$ coprime to $S$ representing these generators. For each $E_0 \in \cE_0$, choose $m_{E_0} \in K^\times$ such that $E_0 \cO_{\widetilde{S}} = m_{E_0} \cO_{\widetilde{S}}$. 

Let $\cE$ be a complete set of representatives of $R_{\c}$ of the form $\prod_{E_0 \in \cE_0} E_0^{n_{E_0}}$, for integers $0 \leq n_{E_0}  < n$. For each element $E = \prod_{E_0 \in \cE_0} E_0^{n_{E_0}} \in \cE$, let $m_E := \prod_{E_0 \in \cE_0} m_{E_0}^{n_{E_0}} \in K^\times$. Note in particular that $E \cO_{\widetilde{S}} = m_{E} \cO_{\widetilde{S}}$ for every $E \in \cE$. Assume without loss of generality that $\cO \in \cE$ and $m_\cO = 1$. 

\subsubsection{Definition of $\chi_\a$}

For coprime $\a, \b \in I^*(S)$, note that we can decompose $\a = (x) E \g^n$ for some $x \in K^\times$ with $x \equiv 1 \pmod{\c}$, $E \in \cE$, and $\g \in I^*(S)$ with $(\g, \b) = 1$. Then we let $\chi_\a(\b) := \big(\frac{xm_E}{\b}\big)_n$, which is well-defined since any other such decomposition $\a = (x') E' \g'^n$ implies $E' = E$ and $\big(\frac{xm_E}{\b}\big)_n = \big(\frac{x'm_{E}}{\b}\big)_n$ as in \cite{FHL03}*{Proposition 1.1}. In the complementary case of $\b \subseteq \cO$ with $(\b, \a \c) \neq 1$, we define $\chi_\a(\b) = 0$. Observe that the construction of the character $\chi_\a$ is not canonical, as it depends on the choices made in this section. These choices are suppressed from the notation.

\subsection{Examples}\label{sec:examples} We briefly pause to compute $\Xi_n(A, B)$ for certain choices of data, showing how to obtain the familiar classical versions of the cubic and quartic large sieves. The computations for this subsection were performed in SageMath, using code generated by ChatGPT 5.6 Sol.

\begin{example}[Cubic large sieve]\label{ex:cubic}
    Let $n=3$ and $K = \Q(\zeta_3) = \Q(\sqrt{-3})$, which has class number $1$. We may take $S = \{(\sqrt{-3})\}$ and $\c = (9) = (\sqrt{-3})^4$. Then $H_\c \simeq (\Z/3\Z)^2 \simeq R_\c$, and we may choose $\cE = \bigcup_{i, j=0}^2\{E_{ij}\}$ for $E_{ij} = (m_{E_{ij}})$ with $m_{E_{ij}} = (-2)^i (1+3\zeta_3)^j$. Note that $m_{E_{ij}} \equiv 1 \pmod{3}$. With this data, writing each $\a, \b \in I(S)$ as $\a = (a)$ and $\b = (b)$ for principal generators $a, b \equiv 1 \pmod{3}$, and choosing $\g = (1)$ in the decomposition $\a = (x) E \g^n$ for $x \equiv 1 \pmod{9}$ and $E \in \cE$, we obtain $\chi_\a(\b) := \big(\frac{x m_E}{\b}\big)_3 = \big(\frac{a}{b}\big)_3$. Hence 
    \begin{equation*}
        \Xi_3(A, B) = \Xi^{\text{HB}}_3(A, B) := \sup_{\normt{\lambda} = 1} \sum_{\substack{a \in \Z[\zeta_3] \\ a \equiv 1 \pmod{3} \\ \cN(a) \leq A}} \mu^2(a) \Big| \sum_{\substack{b \in \Z[\zeta_3] \\ b \equiv 1 \pmod{3} \\ \cN(b) \leq B}} \mu^2(b) \lambda_b \Big(\frac{a}{b}\Big)_3 \Big|^2.
    \end{equation*}
    This is precisely the operator norm considered by Heath-Brown in his version of the cubic large sieve. Thus \cref{thm:main} allows us to obtain the same bound as \cite{HB00}*{Theorem 2}, but relaxing the restriction that $a, b \in \Z[\zeta_3]$ be squarefree to cube-free.
\end{example}

\begin{example}[Quartic large sieve]\label{ex:quartic}
    Let $n = 4$ and $K = \Q(\zeta_4) = \Q(i)$, which has class number $1$. We may take $S = \{(1+i)\}$ and $\c = (1+i)^7$. Then $H_\c \simeq (\Z/4\Z)^2 \simeq R_\c$, and we may choose $\cE = \bigcup_{j, k=0}^3\{E_{jk}\}$ for $E_{jk} = (m_{E_{jk}})$ with $m_{E_{jk}} = (-1+2i)^j (-1-2i)^k$. Note that $m_{E_{jk}} \equiv 1 \pmod{(1+i)^3}$. With this data, writing each $\a, \b \in I(S)$ as $\a = (a)$ and $\b = (b)$ for principal generators $a, b \equiv 1 \pmod{(1+i)^3}$, and choosing $\g = (1)$ in the decomposition $\a = (x) E \g^n$ for $x \equiv 1 \pmod{(1+i)^7}$ and $E \in \cE$, we obtain $\chi_\a(\b) := \big(\frac{x m_E}{\b}\big)_4 = \big(\frac{a}{b}\big)_4$. Thus 
    \begin{equation*}
        \Xi_4(A, B) = \Xi_4^{\text{DDH}}(A, B) := \sup_{\normt{\lambda} = 1} \sum_{\substack{a \in \Z[i] \\ a \equiv 1 \pmod{(1+i)^3} \\ \cN(a) \leq A}} \mu^2(a) \Big| \sum_{\substack{b \in \Z[i] \\ b \equiv 1 \pmod{(1+i)^3} \\ \cN(b) \leq B}} \mu^2(b) \lambda_b \Big(\frac{a}{b}\Big)_4 \Big|^2.
    \end{equation*} 
    Here we recover the classical version of the quartic large sieve, which is important in applications. Hence the bound of \cref{thm:main} improves \cite{CDD26}*{Theorem 4.14}, and asymptotically matches the lower bound \cite{DDH26}*{Theorem 1.2} in logarithmic scale.
\end{example}

\subsection{Properties of \texorpdfstring{$n$}{}-th order Hecke characters}

\subsubsection{Basic properties of $\chi_\a$} Let $\a\in I(S)$. By the discussion in \cite{GL13}*{p.~196--197}, it follows that $\chi_\a$ is a Hecke character modulo $\c\a$ of order dividing $n$, with trivial infinity type. Its conductor is $\c_\a \a^\star$ for some integral ideal $\c_\a \mid \c$, where recall that $\a^\star = \prod_{n \nmid \ord_v(\a)} \p_v$. Furthermore, for any $\a_1, \a_2 \in I^*(S)$ we have
\begin{equation}\label{eq:chi_a_multiplicativity}
    \chi_{\a_1 \a_2}(\b) = \chi_{\a_1}(\b) \cdot \chi_{\a_2}(\b) \qquad \text{for $\b \in I^*(S)$ with $(\b, \a_1) = (\b, \a_2) = 1$.}
\end{equation}
Under the stronger assumption $\a_1, \a_2 \in I(S)$, the relation \eqref{eq:chi_a_multiplicativity} holds for every $\b \subseteq \cO$, since if $(\b, \a_1\a_2\c)\neq 1$ then both sides are $0$ by definition. Denote the primitive Hecke character inducing $\chi_\a$ by $\tchi_\a$. Note that it still satisfies the relation \eqref{eq:chi_a_multiplicativity}.

\subsubsection{Reciprocity} 

By \cite{FHL03}*{Proposition 1.2}, for any coprime $\a, \b \in I^*(S)$ we have that the quantity
\begin{equation}\label{eq:reciprocity}
    \chi_\a(\b) \chi_\b(\a)^{-1} =: \alpha([\a], [\b])
\end{equation} 
depends only on the classes $[\a], [\b] \in R_\c$. This follows from the reciprocity law, see for instance \cite{CF67}*{Exercise 2.10}.

\subsubsection{Conductor of product}

We will make extensive use of the following important property of products of $n$-th order Hecke characters, which is proved as in \cite{GL13}*{p.~197}.

\begin{lemma}[Conductors]\label{lem:conductors}
    If $\a, \b \in I(S)$ are coprime and $[\a] = [\b]$ in $R_\c$, then $\tchi_\a \cdot \overline{\tchi_\b}$ is a Hecke character with trivial infinity type and conductor $\a^\star\b^\star$.
\end{lemma}

\begin{proof}
    Since $\tchi_\a$ and $\overline{\tchi_\b}$ are primitive of conductors respectively $\c_\a \a^\star$ and $\c_\b \b^\star$, and $(\a^\star, \b^\star) = 1$, we conclude that $\tchi_\a \cdot \overline{\tchi_\b}$ has conductor $\f := \c' \a^\star \b^\star$ for some $\c' \mid \c$. It then suffices to show that $\tchi_\a \cdot \overline{\tchi_\b}((x)) = 1$ for all $x \in K^\times$ with $x \equiv 1 \pmod{\a\b}$ and $(x, \c)=1$, since this implies $\f \mid \a \b$, hence $\c' = \cO$. To prove this last assertion, decompose $\a = (x_\a) E_\a \g_\a^n$ and $\b = (x_\b) E_\b \g_\b^n$ for $x_\a, x_\b \in K^\times$ with $x_\a, x_\b \equiv 1 \pmod{\c}$, $E_\a, E_\b \in \cE$, and $\g_\a, \g_\b \in I^*(S)$ with $(\g_\a, x) = (\g_\b, x)=1$. Then $E_\a = E_\b$ since $[\a] =[\b]$ in $R_\c$, therefore
    \begin{equation*}
        \tchi_\a \cdot \overline{\tchi_\b}((x)) = \Big(\frac{x_\a / x_\b}{x}\Big)_n \Big(\frac{m_{E_\a}/m_{E_\b}}{x}\Big)_n = \Big(\frac{x_\a / x_\b}{x}\Big)_n.
    \end{equation*}
    From $x_\a, x_\b \equiv 1 \pmod{\c}$ we conclude that $[x_\a] = [x_\b] = [\cO]$, hence by reciprocity we have the desired relation
    \begin{equation*}
        \Big(\frac{x_\a / x_\b}{x}\Big)_n = \Big(\frac{x}{x_\a / x_\b}\Big)_n = \Big(\frac{x}{\a}\Big)_n\Big(\frac{x}{\b}\Big)_n^{-1} = 1.
    \end{equation*}
\end{proof}

\section{The functional equation and root numbers}

\subsection{Poisson summation formula}

Let $W : (0, \infty) \to [0, \infty)$ be a fixed smooth function with compact support satisfying $W(x) \geq \bbone_{[1, 2]}(x)$. Denote its Mellin transform by $\widehat{W}$ and for $c>0$ let\footnote{Here we correct a typo in \cite{BGL14}*{p.~1933}, which was identified by ChatGPT 5.6 Sol.}
\begin{equation}\label{eq:W_dot_def}
    \dot{W}(x) := \frac{1}{2\pi i} \int_{(c)} \widehat{W}(1-s) \Big(\frac{\Gamma(s)}{\Gamma(1-s)}\Big)^{d/2} \Big(\frac{|D_K|}{(2\pi)^d}\Big)^{s-1/2} x^{-s} \, ds,
\end{equation}
where $d$ and $D_K$ are respectively the degree and discriminant of $K$. Shifting the contour leads to
\begin{equation}\label{eq:W_dot_bound}
    \dot{W}(x) \ll_{W, K, \ell} |x|^{-\ell}
\end{equation}
for any $\ell>0$ and $x > 0$. Using the functional equation for Hecke $L$-functions, as in  \cite{BGL14}*{Lemma 2.2} and \cite{GL13}*{Lemma 3.2}, leads to the following\footnote{This uses that $K$ is totally complex, as it contains $\Q(\zeta_n)$ for $n\geq 3$.}.

\begin{lemma}[Poisson summation formula]\label{lem:general_Poisson_formula}
    Let $\chi$ be a non-trivial primitive ray class character of $K$ with conductor $\f$. For any $A>0$ and $W$ as above, we have
    \begin{equation*}
        \sum_{\a \subseteq \cO} W\Big(\frac{\cN(\a)}{A}\Big) \chi(\a) = \frac{\epsilon(\chi) A}{\sqrt{\cN(\f)}} \sum_{\a \subseteq \cO} \dot{W}\Big(\frac{\cN(\a)}{\cN(\f) / A}\Big) \overline{\chi(\a)},
    \end{equation*}
    where the root number $\epsilon(\chi)$ has absolute value $1$.
\end{lemma}

\subsection{Root number splitting}

We need some information about root numbers. This is given in \cref{lem:root_number} below, whose proof is similar to \cite{BGL14}*{Lemma 2.3}. Before stating it, let us establish some notation. 

Any $\a \in I(S)$ which is $n$-th power free may be uniquely decomposed as\footnote{Note that the squarefree kernel $\a^\star$ coincides with its previous definition.}
\begin{equation}\label{eq:ideal_decomposition}
    \a = \a^\star \cdot \a^\sharp
\end{equation}
for squarefree $\a^\star \in I(S)$ and $\a^\sharp \in I(S)$ such that $\a^\sharp \mid (\a^\star)^\infty$ and $\a^\star \cdot \a^\sharp$ is $n$-th power free. 

Since $(\a, \c) = 1$, the primitive characters $\tchi_{\a}$, which we saw have conductor $\c_{\a} \a^\star$ for some $\c_{\a} \mid \c$, can be decomposed by CRT as\footnote{Note that $\chi^{(\a)}$ and $\psi_\a$ need not be ray class characters (but do have opposite unitary infinity types). Fixing this infinity type, the decomposition \eqref{eq:coprime_decomp} is unique up to multiplication by a class group character.}
\begin{equation}\label{eq:coprime_decomp}
    \tchi_{\a} = \chi^{(\a)} \cdot \psi_{\a}
\end{equation}
for primitive unitary Hecke characters $\chi^{(\a)}$ and $\psi_{\a}$ of conductors respectively $\a^\star$ and $\c_{\a}$. 

\begin{definition}\label{def:primitive_product}
    For any $\a, \b \in I(S)$, denote the primitive ray class character that induces $\tchi_{\a} \cdot \overline{\tchi_{\b}}$ by $\tchi_{\a, \b}$.
\end{definition}

If $\a, \b\in I(S)$ are coprime and $[\a] = [\b]$, then \cref{lem:conductors} implies that $\tchi_{\a, \b}$ has conductor $\a^\star \b^\star$, which is coprime with $\c$. Hence $\psi_{\a} \cdot \overline{\psi_{\b}}$ has conductor $\cO$. Thus $\c_{\a} = \c_{\b}$, and furthermore we may modify the decompositions \eqref{eq:coprime_decomp}, multiplying by a unitary Hecke character of conductor $\cO$ so that $\psi_{\a} = \psi_{\b}$ whenever $[\a] = [\b]$, by fixing a global representative for each class of $R_\c$. With this choice of representatives, \textbf{which we fix once and for all}, $\psi_{\a} \cdot \overline{\psi_{\b}}$ is trivial on $I(S)$, so 
\begin{equation}\label{eq:_primitive_chi_decomposition}
    \tchi_{\a, \b} = \chi^{(\a)} \cdot \overline{\chi^{(\b)}} \qquad \text{ for } \a, \b \in I(S) \text{ with } (\a, \b)=1 \text{ and } [\a] = [\b].
\end{equation}

\begin{lemma}[Root number of product]\label{lem:root_number}
    There exists a fixed sequence $\{\beta_{\b}\}_{\b \in I(S)}$ satisfying $|\beta_{\b}| = 1$ such that the following holds. For any $n$-th power free, coprime $\b_1, \b_2 \in I(S)$ with $[\b_1] = [\b_2] =: \cC \in R_\c$, the primitive ray class character $\tchi_{\b_1, \b_2}$ has root number
    \begin{equation*}
        \epsilon(\tchi_{\b_1, \b_2}) = \eta \cdot  \beta_{\b_1} \cdot \overline{\beta_{\b_2}} \cdot \chi_{\b_1^\sharp}(\b_2^\star) \cdot \overline{\chi_{\b_2^\sharp}(\b_1^\star)}
    \end{equation*}
    for some $\eta \in \C$ satisfying $|\eta|=1$ and depending only on $[\b_1^\star], [\b_2^\star], \cC \in R_\c$.
\end{lemma}

\begin{proof}
    Since $[\b_1] = [\b_2]$, the previous discussion culminating in \eqref{eq:_primitive_chi_decomposition} gives
    \begin{equation*}
        \tchi_{\b_1, \b_2} = \chi^{(\b_1)} \cdot \overline{\chi^{(\b_2)}}.
    \end{equation*}
    
    Since $\epsilon(\overline{\xi}) = \overline{\epsilon(\xi)}$, by \cite{Iwa97}*{Chapter 12, (12.23) and (12.26)} and \eqref{eq:coprime_decomp} we conclude that\footnote{The following typos should be corrected in \cite{Iwa97}*{Theorem 12.3}, according to \cite{Lan94}*{XIV \S8}. The Gamma factors $\Gamma\big(\frac{|u_\sigma| + n_\sigma(s + iv_\sigma)}{2}\big)$ should be $\Gamma\big(\frac{|u_\sigma| + n_\sigma s - iv_\sigma}{2}\big)$. We then have $\Lambda(s, \xi) = \epsilon(\xi) \Lambda(1-s, \overline{\xi})$ for $\epsilon(\xi) = 2^{-iv_\C} i^{-u} W(\xi) N(\frakm)^{-1/2}$, where $u = \sum_\sigma |u_\sigma|$, $v_\C = \sum_{\sigma \text{ complex}} v_\sigma$, and $W$ is the same as in \cite{Iwa97}*{(12.21)}. In our case, since $K$ is totally complex, the normalization $\sum_\sigma v_\sigma = 0$ gives $v_\C = 0$. Note the absolute values in the definition of $u$, which were missing and are relevant here. These typos were found by ChatGPT 6 Astra.} 
    \begin{align*}
        \epsilon(\tchi_{\b_1, \b_2}) & = \epsilon(\chi^{(\b_1)} \cdot \overline{\chi^{(\b_2)}}) = i^{u_{\b_1} + u_{\b_2}}\cdot\epsilon(\chi^{(\b_1)}) \cdot \overline{\epsilon(\chi^{(\b_2)})} \cdot \chi^{(\b_1)}(\b_2^\star) \cdot \overline{\chi^{(\b_2)}(\b_1^\star)} \\
        & = i^{u_{\b_1} + u_{\b_2}}\cdot \epsilon(\chi^{(\b_1)}) \cdot \overline{\epsilon(\chi^{(\b_2)})} \cdot \tchi_{\b_1}(\b_2^\star) \cdot \overline{\tchi_{\b_2}(\b_1^\star)} \cdot \overline{\psi_{\b_1}(\b_2^\star)} \cdot \psi_{\b_2}(\b_1^\star).
    \end{align*}
    Here $u_{\b_i}$ depends only on the infinity type of $\chi^{(\b_i)}$, which is opposite that of $\psi_{\b_i}$. Those are fixed and depend only on $\cC$, so denote $u_{\b_1} = u_{\b_2} =: u_\cC$. Multiplying by $\psi_{\b_1} \cdot \overline{\psi_{\b_2}} (\b_1^\star\b_2^\star) = 1$,
     \begin{align*}
        \epsilon(\tchi_{\b_1, \b_2}) & = (-1)^{u_\cC}\cdot \epsilon(\chi^{(\b_1)}) \cdot \overline{\epsilon(\chi^{(\b_2)})} \cdot \chi_{\b_1}(\b_2^\star) \cdot \overline{\chi_{\b_2}(\b_1^\star)} \cdot \psi_{\b_1}(\b_1^\star) \cdot \overline{\psi_{\b_2}(\b_2^\star)}.
    \end{align*}
    By \eqref{eq:chi_a_multiplicativity} and \eqref{eq:reciprocity} we have
    \begin{align*}
        \chi_{\b_1}(\b_2^\star) \cdot \overline{\chi_{\b_2}(\b_1^\star)} & = \chi_{\b_1^\star}(\b_2^\star) \cdot \overline{\chi_{\b_2^\star}(\b_1^\star)} \cdot \chi_{\b_1^\sharp}(\b_2^\star) \cdot \overline{\chi_{\b_2^\sharp}(\b_1^\star)}  = \alpha([\b_1^\star], [\b_2^\star]) \cdot \chi_{\b_1^\sharp}(\b_2^\star) \cdot \overline{\chi_{\b_2^\sharp}(\b_1^\star)}.
    \end{align*}
    The result follows after setting $\beta_{\b} := \epsilon(\chi^{(\b)}) \cdot \psi_\b(\b^\star)$ and $\eta := (-1)^{u_\cC}\cdot \alpha([\b_1^\star], [\b_2^\star])$.
\end{proof}

\begin{remark}
    Note that \cref{lem:root_number} is notably different from its analogue for squarefree $\b_1, \b_2$ in \cite{BGL14}*{Lemma 2.3}, where the root number splits completely as a product of functions of each of $\b_1$ and $\b_2$. We must contend with the cross terms $\chi_{\b_1^\sharp}(\b_2^\star) \cdot \overline{\chi_{\b_2^\sharp}(\b_1^\star)}$, which will cause technical difficulties. Fortunately these difficulties are dampened by conductor dropping when $\b_1^\sharp \b_2^\sharp$ is large.
\end{remark}


\section{Preliminary lemmas}

This section gathers some standard results which will be useful throughout the paper. We fix $n\geq 3$, $K \supseteq \Q(\zeta_n)$, the function $W$, and the data of \cref{sec:n_th_order_construction} for constructing $n$-th order characters from now on, so implied constants may depend on them. 

For $\a, \b \in I(S)$, denote the condition $A \leq \cN(\a) < 2A$ by $\a \sim A$, and let $\tchi_{\a, \b}$ be the primitive character inducing $\tchi_\a \cdot \overline{\tchi_{\b}}$, as in \cref{def:primitive_product}. Define
\begin{equation*}
    \Sigma_1(A, B, \lambda) := \sumflat_{\a \sim A} \Big| \sumflat_{\b \sim B} \lambda_\b \chi_\a(\b) \Big|^2 \qquad \text{and} \qquad  \Sigma_2(A, B, \lambda) := \sum_{\substack{\a \in I(S) \\ \a \sim A}} \Big| \sumflat_{\b \sim B} \lambda_\b \chi_\b(\a) \Big|^2.
\end{equation*}
For $\Vec{\cC} = (\cC_1, \cC_2, \cC) \in (R_\c)^{\oplus 3}$, set
\begin{equation*}
    \Sigma_3(A, B, \lambda, \Vec{\cC}) := \Bigg| \sumflat_{\substack{\b_1, \b_2 \sim B \\ (\b_1, \b_2)=1 \\ [\b_1^\star] = \cC_1,\, [\b_2^\star] = \cC_2 \\ [\b_1] = [\b_2] = \cC}} \lambda_{\b_1} \overline{\lambda_{\b_2}} \sum_{\a \subseteq \cO} W\Big(\frac{\cN(\a)}{A}\Big) \tchi_{\b_1, \b_2}(\a) \Bigg|.
\end{equation*}
Let
\begin{equation*}
    \sL_i(A, B) := \sup_{\normt{\lambda}=1} \Sigma_{i}(A, B, \lambda) \qquad\text{ for } i \in \{1, 2\}
\end{equation*}
and
\begin{equation*}
    \sL_3(A, B) := \sup_{\normt{\lambda}=1} \sup_{\Vec{\cC} \in (R_\c)^{\oplus 3}} \Sigma_{3}(A, B, \lambda, \Vec{\cC}).
\end{equation*}

\begin{lemma}[Symmetry]\label{lem:symmetry}
    For any $A, B \geq 1$, we have $\sL_1(A,B) \asymp \sL_1(B,A)$.
\end{lemma}

\begin{proof}
    The proof is virtually identical to \cite{GL13}*{Lemma 4.5}. Indeed, denoting $\lambda^\cC_\a := \lambda_\a \cdot \alpha([\a], \cC)$, the reciprocity relation  \eqref{eq:reciprocity} implies
    \begin{align}
        \sumflat_{\b \sim B} \Big|\sumflat_{\a \sim A} \lambda_\a \chi_\a(\b) \Big|^2 & = \sum_{\cC \in R_\c} \sumflat_{\substack{\b \sim B \\ [\b] = \cC}} \Big|\sumflat_{\a \sim A} \lambda^\cC_\a \chi_\b(\a) \Big|^2 \leq \sum_{\cC \in R_\c} \sumflat_{\substack{\b \sim B}} \Big|\sumflat_{\a \sim A} \lambda^\cC_\a \chi_\b(\a) \Big|^2 \label{eq:example_reciprocity} \\
        & \leq \sum_{\cC \in R_\c} \sL_1(B, A) \cdot \sumflat_{\a \sim A} |\lambda^\cC_\a|^2 = |R_\c| \cdot \sL_1(B, A) \cdot \sumflat_{\a \sim A} |\lambda_\a|^2. \nonumber
    \end{align}
    The result then follows from the duality principle for the large sieve.
\end{proof}

\begin{lemma}[Monotonicity]\label{lem:monotonicity}
    There exists a constant $C \geq 1$ such that for any $A_1, A_2, B \geq 1$ with $A_2 \geq C A_1 \log(2 A_1 B)$ we have
    \begin{equation*}
        \sL_1(A_1, B) \ll \sL_1(A_2, B).
    \end{equation*}
\end{lemma}

\begin{proof}
    The proof is very similar to \cite{HB95}*{Lemma 9}, but we give full details. Let $P := \frac{A_2}{A_1}$ and consider prime ideals $\p \in I(S)$ with $\frac{2}{3}P < \cN(\p) < \frac{4}{3}P$. We bound $\Sigma_1(A_1, B, \lambda)$ using
     \begin{equation}\label{eq:AM-GM}
         \sumflat_{\substack{\a}} \Big| \sumflat_{\substack{\b}} \Big|^2 \leq \sumflat_{\substack{\p \mid \a}} \Big| \sumflat_{\substack{\b}} \Big|^2 + 2 \sumflat_{\substack{\p \nmid \a}} \Big| \sumflat_{\substack{\p \nmid \b}} \Big|^2 + 2 \sumflat_{\substack{\p \nmid \a}} \Big| \sumflat_{\substack{\p \mid \b}} \Big|^2.
     \end{equation}
     Consider the cases where $\cN(\a)$ is in $[A_1, \frac{3}{2}A_1)$ or $[\frac{3}{2}A_1, 2 A_1)$, restricting to $\p$ with $\cN(\p)$ respectively in $(P, \frac{4}{3}P)$ and $(\frac{2}{3}P, P)$. 

     By \eqref{eq:chi_a_multiplicativity} we have
     \begin{equation*}
         \sumflat_{\substack{\alpha A_1 \leq \cN(\a) < \beta A_1 \\ \p \nmid \a}} \Big| \sumflat_{\substack{\b \sim B \\ \p \nmid \b}} \lambda_\b \chi_\a(\b) \Big|^2 \leq \sumflat_{\substack{\frakr \sim A_2}} \Big| \sumflat_{\substack{\b \sim B \\ \p \nmid \b}} \lambda_\b \overline{\chi_{\p}(\b)} \cdot  \chi_{\frakr}(\b) \Big|^2 \leq \sL_1(A_2, B) \cdot \sum_{\b\sim B} |\lambda_\b|^2
     \end{equation*}
     for $(\alpha, \beta) = (1, \frac{3}{2})$ or $(\frac{3}{2}, 2)$, since $\frakr := \p\a \in I(S)$ is $n$-th power free and has norm in $[A_2, 2A_2)$. Thus for each $(\alpha, \beta)$ and $\p \in I(S)$ in the appropriate range, \eqref{eq:AM-GM} yields
     \begin{align*}
         \sumflat_{\substack{\alpha A_1 \leq \cN(\a) < \beta A_1}} \Big| \sumflat_{\substack{\b \sim B}} \lambda_\b \chi_\a(\b) \Big|^2 \ll \sumflat_{\substack{\a \sim A_1 \\ \p \mid \a}} \Big| \sumflat_{\substack{\b \sim B}} \lambda_\b \chi_\a(\b) \Big|^2 &+ \sL_1(A_2, B) \cdot \sum_{\b\sim B} |\lambda_\b|^2 \\
         &+ \sL_1(A_1, B) \cdot \sum_{\substack{\b\sim B \\ \p \mid \b}} |\lambda_\b|^2.
     \end{align*} 

     We sum over $\p \in I(S)$ in the relevant range. Observe (by taking $C \geq 1$ sufficiently large) that there are $\asymp\frac{P}{\log{P}}$ such prime ideals $\p$, and the ideals $\a$ and $\b$ have $\ll \frac{\log(2A_1B)}{\log{P}}$ prime factors among those $\p$. Hence
     \begin{align*}
          & \frac{P}{\log{P}} \cdot \sumflat_{\substack{\alpha A_1 \leq \cN(\a) < \beta A_1}} \Big| \sumflat_{\substack{\b \sim B}} \lambda_\b \chi_\a(\b) \Big|^2 \\
          \ll \ & \frac{\log(2A_1B)}{\log{P}} \cdot \sumflat_{\substack{\a \sim A_1}} \Big| \sumflat_{\substack{\b \sim B}} \lambda_\b \chi_\a(\b) \Big|^2 + \frac{P}{\log{P}} \cdot \sL_1(A_2, B) \cdot \sum_{\b\sim B} |\lambda_\b|^2 \\
          & \qquad \qquad \qquad \qquad \qquad \qquad + \frac{\log(2A_1B)}{\log{P}} \cdot \sL_1(A_1, B) \cdot \sum_{\substack{\b\sim B}} |\lambda_\b|^2 \\
          \ll \ & \Big( \frac{\log(2A_1B)}{\log{P}} \cdot \sL_1(A_1, B) + \frac{P}{\log{P}} \cdot \sL_1(A_2, B) + \frac{\log(2A_1B)}{\log{P}} \cdot \sL_1(A_1, B) \Big)\cdot  \sum_{\substack{\b\sim B}} |\lambda_\b|^2.
     \end{align*}

     Choosing $C \geq 1$ sufficiently large, this implies
     \begin{equation*}
         \Sigma_1(A_1, B, \lambda) \ll \sL_1(A_2, B) \cdot \normt{\lambda}^2,
     \end{equation*}
     which gives the desired result.
\end{proof}

We need relations between the $\sL_i$ for $i \in \{1, 2, 3\}$. Bounding $\sL_1$ by $\sL_2$ is very simple.

\begin{lemma}[Completing the $\a$ sum]\label{lem:K_1_K_2_inequality}
    For any $A, B \geq 1$ we  have
    \begin{equation*}
        \sL_1(A, B) \ll \sL_2(A, B).
    \end{equation*}
\end{lemma}

\begin{proof}
    Follows directly from the reciprocity relation \eqref{eq:reciprocity}, after separating the sum over $\a$ into classes of $R_\c$, similarly to \eqref{eq:example_reciprocity}, and then extending it to all $\a \in I(S)$ by positivity.
\end{proof}

Bounding $\sL_2$ by $\sL_3$ is also relatively simple, despite some technicalities.

\begin{lemma}[GCD sorting]\label{lem:GCD_sorting}
    For any $A, B \geq 1$, there exist $1 \leq \Delta_1 \ll \Delta_2$ such that
    \begin{equation*}
        \sL_2(A, B) \ll_\varepsilon B^\varepsilon \cdot \sL_3\Big( \frac{A}{\Delta_1}, \frac{B}{\Delta_2}\Big)
    \end{equation*}
    for every $\varepsilon > 0$.
\end{lemma}

\begin{proof}
    The proof is very similar to \cite{BGL14}*{Lemma 3.2}, which in turn closely resembles \cite{HB00}*{Lemma 7}. We fix a sequence $\{\lambda_\b\}_{\b \subseteq \cO}$ with $\normt{\lambda}=1$. By Cauchy--Schwarz,
    \begin{equation*}
        \sum_{\substack{\a \in I(S) \\ \a \sim A}} \Big| \sumflat_{\b \sim B} \lambda_\b \chi_\b(\a) \Big|^2 \ll \max_{\cC \in R_\c} \sum_{\substack{\a \in I(S) \\ \a \sim A}} \Big| \sumflat_{\substack{\b \sim B \\ [\b] = \cC}} \lambda_\b \chi_\b(\a) \Big|^2.
    \end{equation*}
    
    Inserting the smooth non-negative weight $W(\cN(\a)/A)$, opening the square, and sorting by $\d = (\b_1, \b_2)$, the display above is
    \begin{equation}\label{eq:K_3_intermediate}
        \ll \max_{\cC \in R_\c} \sumflat_{\d} \Bigg| \sum_{\substack{\a \in I(S)}} W\Big(\frac{\cN(\a)}{A}\Big)  \sumflat_{\substack{\b_1, \b_2 \sim B/\cN(\d) \\ (\b_1, \b_2)=1 \\ [\b_1] = [\b_2] = \cC [\d]^{-1} \\ \b_1\d, \b_2\d \text{ $n$-th free}}} \lambda_{\b_1 \d} \overline{\lambda_{\b_2 \d}} \cdot \chi_{\b_1\d}(\a) \overline{\chi_{\b_2\d}(\a)}\Bigg| .
    \end{equation}
    Here we can replace $\chi_{\b_1\d}(\a) \overline{\chi_{\b_2\d}(\a)}$ by $\chi_{\b_1}(\a) \overline{\chi_{\b_2}(\a)} \bbone_{(\a, \c\d)=1}$ using \eqref{eq:chi_a_multiplicativity}. Since $\a \in I(S)$, this in turn can be replaced by $\tchi_{\b_1}(\a) \overline{\tchi_{\b_2}(\a)} \bbone_{(\a, \d)=1}$, since either $(\a, \b_1\b_2)=1$ or both expressions are $0$, as $\b_1$ and $\b_2$ are $n$-th power free. Finally, this last expression can be replaced by $\tchi_{\b_1, \b_2}(\a) \bbone_{(\a, \d)=1}$, since $\a \in I(S)$ and $(\b_1, \b_2)=1$. Thus \eqref{eq:K_3_intermediate} is
    \begin{equation}\label{eq:K_3_intermediate_2}
        \ll \sumflat_{\d} \max_{\cC \in R_\c} \Bigg| \sum_{\substack{\a \in I(S) \\ (\a, \d)=1}} W\Big(\frac{\cN(\a)}{A}\Big)  \sumflat_{\substack{\b_1, \b_2 \sim B/\cN(\d) \\ (\b_1, \b_2)=1 \\ [\b_1] = [\b_2] = \cC \\ \b_1\d, \b_2\d \text{ $n$-th free}}} \lambda_{\b_1 \d} \overline{\lambda_{\b_2 \d}} \cdot \tchi_{\b_1, \b_2}(\a)\Bigg|.
    \end{equation}

    We now remove the coprimality condition $(\a, \c\d)=1$ by M\"obius inversion. Recalling the decomposition $\tchi_{\b_1, \b_2} = \chi^{(\b_1)} \cdot \overline{\chi^{(\b_2)}}$ of \eqref{eq:_primitive_chi_decomposition} and denoting
    \begin{equation*}
        \lambda_\b^\dagger = \lambda_\b^\dagger(\d, \r) := \lambda_{\b\d} \cdot \chi^{(\b)}(\r) \cdot \bbone_{\b\d \text{ $n$-th free}},
    \end{equation*}
    observe that \eqref{eq:K_3_intermediate_2} is
    \begin{align*}
        & = \sumflat_{\d} \max_{\cC \in R_\c} \Bigg| \sum_{\r \mid \c\d} \mu(\r) \sum_{\substack{\a \subseteq \cO}} W\Big(\frac{\cN(\a)}{A /\cN(\r)}\Big)  \sumflat_{\substack{\b_1, \b_2 \sim B/\cN(\d) \\ (\b_1, \b_2)=1 \\ [\b_1] = [\b_2] = \cC}} \lambda^\dagger_{\b_1} \overline{\lambda^\dagger_{\b_2}} \cdot \tchi_{\b_1, \b_2}(\a) \Bigg| \\
        & \ll \sumflat_{\d} \sum_{\r \mid \c\d}  \sL_3\Big(\frac{A}{\cN(\r)}, \frac{B}{\cN(\d)}\Big) \sum_{\substack{\b \sim B/\cN(\d)}} |\lambda^\dagger_{\b}|^2 \ll \max_{1\leq \Delta_1 \ll \Delta_2} \sL_3\Big(\frac{A}{\Delta_1}, \frac{B}{\Delta_2}\Big) \sum_{\d} \sum_{\r \mid \c\d} \sum_{\substack{\b \sim B/\cN(\d)}} |\lambda_{\b\d}|^2.
    \end{align*}
    By the divisor bound, the last three sums over $\d, \r, \b$ are $\ll_\varepsilon B^\varepsilon \sum_\b |\lambda_\b|^2 = B^\varepsilon$, so the result follows.
\end{proof}

In the next two sections, we bound $\sL_3$ in terms of $\sL_2$ using the Poisson summation formula, and $\sL_2$ in terms of $\sL_1$ by removing $n$-th powers.


\section{Application of Poisson summation}

To relate $\sL_3$ back to $\sL_2$, we apply Poisson summation as in \cites{HB00, GL13, BGL14}. However, a new issue arises: since here $\b_1, \b_2 \in I(S)$ are not assumed to be squarefree, the root number of $\tchi_{\b_1, \b_2}$ might not split as a product of functions of $\b_1$ and $\b_2$. 

To remedy this issue, we observe that at least it is possible to uncouple the dependence of $\epsilon(\tchi_{\b_1, \b_2})$ on the conductors $\b_1^\star$ and $\b_2^\star$. It is also crucial that $\b_1^\sharp$ and $\b_2^\sharp$ do not interact directly (however there are cross dependencies $\b_1^\star \leftrightarrow \b_2^\sharp$ and $\b_2^\star \leftrightarrow \b_1^\sharp$, whose shape is important for the argument). This maneuver is a bit technical, but turns out to not be wasteful: we will recover the lost factors since the length after Poisson summation is proportional to the conductor $\b_1^\star\b_2^\star$, instead of the larger $\b_1\b_2 = \b_1^\star\b_2^\star\cdot\b_1^\sharp\b_2^\sharp$. 

\begin{lemma}[Poisson summation]\label{lem:Poisson_K_3}
    For any $A, B \geq 1$ and $\varepsilon>0$, we have
    \begin{equation*}
        \sL_3(A, B) \ll_\varepsilon A + (AB)^\varepsilon \max_{\substack{1\leq B_1, B_2 \leq B }} \max_{\substack{1 \leq K_i \leq (AB)^\varepsilon\frac{B^2}{A B_i } \\ 1 \leq L_i \leq \frac{B}{B_i}}} \frac{A (B_1B_2)^{1/2}}{B} \cdot \sL_2(K_1, L_1)^{1/2} \cdot \sL_2(K_2, L_2)^{1/2}.
    \end{equation*}
\end{lemma}

\begin{proof}
    If $B = 1$ then the trivial bound $\sL_3(A, B) \ll A $ suffices, so assume $B > 1$. Let $\{\lambda_\b\}_{\b \subseteq \cO}$ be a sequence with $\normt{\lambda}=1$. By \cref{lem:conductors}, the characters $\tchi_{\b_1, \b_2}$ appearing in $\Sigma_3(A, B, \lambda, \Vec{\cC})$ are primitive of conductor $\b_1^\star \b_2^\star$. Applying the Poisson summation formula of \cref{lem:general_Poisson_formula} to the sum over $\a$, expressing the root number $\epsilon(\tchi_{\b_1, \b_2})$ using \cref{lem:root_number}, and denoting $\lambda'_\b := \lambda_\b \cdot \beta_\b$, we obtain
    \begin{equation*}
        \Sigma_3(A, B, \lambda, \Vec{\cC}) \ll A \sumflat_{\substack{\b_1, \b_2 \sim B \\ (\b_1, \b_2)=1 \\ [\b_1^\star] = \cC_1,\, [\b_2^\star] = \cC_2 \\ [\b_1] = [\b_2] = \cC}} \chi_{\b_1^\sharp}(\b_2^\star) \cdot \overline{\chi_{\b_2^\sharp}(\b_1^\star)} \cdot \lambda'_{\b_1} \overline{\lambda'_{\b_2}} \sum_{\a \subseteq \cO} \dot{W}\Big(\frac{\cN(\a)}{\cN(\b_1^\star \b_2^\star) / A}\Big) \frac{\overline{\tchi_{\b_1, \b_2}(\a)}}{\sqrt{\cN(\b_1^\star \b_2^\star)}}.
    \end{equation*}
    
    Decompose $\a = \s\a'$ for $\a' \in I(S)$ and $\s \mid \c^\infty$, where observe that $\a = (0)$ does not contribute because $\tchi_{\b_1, \b_2}$ is non-trivial, as $B > 1$. Also decompose the sums over $\b_1$ and $\b_2$ according to \eqref{eq:ideal_decomposition} into $\b_i = \b_i^\sharp \cdot \b_i^\star$. Thus for $i \in \{1, 2\}$ we write\footnote{In what follows, the variables $\b_i^\sharp$ and $\b_i^\star$ are free and not necessarily related to an underlying $\b_i$. In other words, the symbols $\sharp$ and $\star$ should not be interpreted as operators, but are only meant to keep the notation consistent with previous steps for the convenience of the reader.}
    \begin{equation*}
        \sum_{0 \neq \a \subseteq \cO} \mapsto \sum_{\s \mid \c^\infty} \sum_{\a' \in I(S)} (\a = \s\a') \qquad \text{and} \qquad \sumflat_{\substack{\b_1, \b_2 \sim B \\ (\b_1, \b_2)=1 \\ [\b_1^\star] = \cC_1,\, [\b_2^\star] = \cC_2 \\ [\b_1] = [\b_2] = \cC}} \mapsto \sumflat_{\substack{[\b_1^\sharp] = \cC \cC_1^{-1} \\ [\b_2^\sharp] = \cC \cC_2^{-1}}} \sumstar_{\substack{\b_i^\star \sim B / \cN(\b_i^\sharp) \\ (\b_1^\star, \b_2^\star)=1 \\ \b_i^\sharp \b_i^\star \text{ $n$-th free} \\ \b_i^\sharp \mid (\b_i^\star)^\infty,\ [\b_i^\star] = \cC_i}} (\b_i = \b_i^\sharp \b_i^\star),
    \end{equation*}
    where recall that $\sumstar_\b$ denotes summation over squarefree $\b \in I(S)$. 
    
    Split into dyadic ranges $\a' \sim K$, $\b_1^\sharp \sim B_1$, and $\b_2^\sharp \sim B_2$, where $1 \leq B_1, B_2 < 2B$ and $K\geq 1$ run over powers of $2$. Denote $Z := (AB)^\varepsilon$, where $\varepsilon > 0$. Observe by the decay of $\dot{W}$ in \eqref{eq:W_dot_bound} that the values $K > Z B^2 / A B_1 B_2 \cN(\s)$ contribute an error $\ll_\varepsilon (AB)^{-100}$. Hence
    \begin{align}
        &\Sigma_3(A, B, \lambda, \Vec{\cC}) \ll_\varepsilon (AB)^{-100} + ZA \sum_{\s \mid \c^\infty} \sup_{\substack{1\leq B_1, B_2 < 2 B \\ 1\leq K \leq  \frac{Z B^2}{A B_1 B_2 \cN(\s)}}} \sumflat_{\substack{\b_1^\sharp \sim B_1 \\ [\b_1^\sharp] = \cC \cC_1^{-1}}} \sumflat_{\substack{\b_2^\sharp \sim B_2 \\ [\b_2^\sharp] = \cC \cC_2^{-1}}} \nonumber \\
        & \times \Bigg| \sumstar_{\substack{\b_i^\star \sim B / \cN(\b_i^\sharp) \\ (\b_1^\star, \b_2^\star)=1 \\ \b_i^\sharp \b_i^\star \text{ $n$-th free} \\ \b_i^\sharp \mid (\b_i^\star)^\infty,\ [\b_i^\star] = \cC_i}} \chi_{\b_1^\sharp}(\b_2^\star) \cdot \overline{\chi_{\b_2^\sharp}(\b_1^\star)} \cdot \lambda'_{\b_1^\sharp \b_1^\star} \overline{\lambda'_{\b_2^\sharp \b_2^\star}} \sum_{\substack{\a' \in I(S) \\ \a' \sim K}} \dot{W}\Big(\frac{\cN(\s\a')}{\cN(\b_1^\star \b_2^\star) / A}\Big)\frac{\overline{\tchi_{\b_1^\sharp \b_1^\star, \b_2^\sharp \b_2^\star}(\s\a')}}{\sqrt{\cN(\b_1^\star \b_2^\star)}}\Bigg|. \qquad \label{eq:interm_Poisson}
    \end{align}

    Since the summation ensures that $\a', \b_1, \b_2 \in I(S)$, $(\b_1, \b_2)=1$, $[\b_1] = [\b_2]$, and $\b_1, \b_2$ are $n$-th power free, where we temporarily denote $\b_i := \b_i^\sharp \b_i^\star$, the identity \eqref{eq:_primitive_chi_decomposition} gives
    \begin{equation*}
        \tchi_{\b_1, \b_2}(\s\a') = \tchi_{\b_1, \b_2}(\s) \cdot \tchi_{\b_1, \b_2}(\a') = \chi^{(\b_1)}(\s) \cdot \overline{\chi^{(\b_2)}(\s)} \cdot \tchi_{\b_1, \b_2}(\a').
    \end{equation*}
     It also follows from \eqref{eq:chi_a_multiplicativity} that
     \begin{equation*}
         \tchi_{\b_1, \b_2}(\a') = \tchi_{\b_1}(\a') \cdot \overline{\tchi_{\b_2}(\a')} = \chi_{\b_1}(\a') \cdot \overline{\chi_{\b_2}(\a')} = \chi_{\b_1^\sharp}(\a') \cdot \overline{\chi_{\b_2^\sharp}(\a')} \cdot \chi_{\b_1^\star}(\a') \cdot \overline{\chi_{\b_2^\star}(\a')}.
     \end{equation*}
     
     We apply the identities above to \eqref{eq:interm_Poisson}, then insert the definition \eqref{eq:W_dot_def} of $\dot{W}$ and shift the line of integration to $\Re(s) = \varepsilon$. The rapid decay of the Mellin transform $\widehat{W}$ allows us to truncate to $|\Im(s)| \leq Z$, once again at the cost of an overall error $\ll_\varepsilon (AB)^{-100}$. Thus apart from this error term, the sum over $\b_1^\star$ and $\b_2^\star$ in \eqref{eq:interm_Poisson} is
    \begin{equation}\label{eq:interm_2_Poisson}
        \ll_\varepsilon Z^{O(1)} \int_{-Z}^Z \sum_{\substack{\a' \in I(S) \\ \a' \sim K}} \Bigg| \sumflat_{\substack{\b_1^\star \sim B / \cN(\b_1^\sharp) \\ \b_2^\star \sim B / \cN(\b_2^\sharp) \\ (\b_1^\star, \b_2^\star)=1}}   \lambda_{\b_1^\sharp}^{\dagger}(\b_1^\star) \cdot \lambda_{\b_2^\sharp}^{\dagger\dagger}(\b_2^\star)  \cdot \frac{\overline{\chi_{\b_2^\sharp}(\b_1^\star)}\cdot \overline{\chi_{\b_1^\star}(\a')} \cdot \chi_{\b_1^\sharp}(\b_2^\star) \cdot \chi_{\b_2^\star}(\a')}{\cN(\b_1^\star \b_2^\star)^{1/2 -\varepsilon + it}} \Bigg| \, |dt|,
    \end{equation} 
    where for $\b \in I(S)$ we set
    \begin{equation*}
        \lambda_{\b_1^\sharp}^\dagger(\b) := \lambda'_{\b_1^\sharp \b} \cdot \overline{\chi^{(\b_1^\sharp \b)}(\s)} \cdot \bbone_{\substack{\b \text{ squarefree} \\ \b_1^\sharp \b \text{ $n$-th free} \\ \b_1^\sharp \mid \b^\infty,\ [\b] = \cC_1}}\quad \text{and} \quad
        \lambda_{\b_2^\sharp}^{\dagger\dagger}(\b) := \overline{\lambda'_{\b_2^\sharp \b}} \cdot \chi^{(\b_2^\sharp \b)}(\s) \cdot \bbone_{\substack{\b \text{ squarefree} \\ \b_2^\sharp \b \text{ $n$-th free} \\ \b_2^\sharp \mid \b^\infty,\ [\b] = \cC_2}}.
    \end{equation*}
    
    The classes $[\b_i^\star] = \cC_i$ and $[\b_j^\sharp] = \cC\cC_j^{-1}$ are fixed for $i, j\in \{1, 2\}$, so by \eqref{eq:reciprocity} we may replace the two cross terms of the form $\chi_{\b_j^\sharp}(\b_i^\star)$ in \eqref{eq:interm_2_Poisson} by $\chi_{\b_i^\star}(\b_j^\sharp)$. Using M\"obius inversion to remove the coprimality condition in  \eqref{eq:interm_2_Poisson} and writing $\b_i^\star = \r \b'_i$, we conclude that $\Sigma_3(A, B, \lambda, \Vec{\cC})$ is
    \begin{equation}\label{eq:endgame_Poisson}
        \ll_\varepsilon 1 + Z^{O(1)} A \sum_{\s \mid \c^\infty} \sup_{\substack{1\leq B_1, B_2, R < 2 B \\ 1\leq K \leq \frac{Z B^2}{A B_1 B_2 \cN(\s)}}} \int_{-Z}^Z \sumflat_{\substack{\b_1^\sharp \sim B_1 \\ [\b_1^\sharp] = \cC \cC_1^{-1}}} \sumflat_{\substack{\b_2^\sharp \sim B_2 \\ [\b_2^\sharp] = \cC \cC_2^{-1}}} \sumstar_{\r\sim R} \sum_{\substack{\a' \in I(S) \\ \a' \sim K}} |\cS_1| \cdot |\cS_2| \, |dt|,
    \end{equation}
    where
    \begin{equation*}
        \cS_1 :=  \sumflat_{\substack{\b'_1 \sim B / \cN(\b_1^\sharp \r)}}   \lambda_{\b_1^\sharp}^{\dagger}(\r \b'_1) \cdot \frac{\overline{\chi_{\b'_1}(\b_2^\sharp \a')}}{\cN(\r\b'_1)^{1/2 -\varepsilon + it}} \quad \text{and} \quad \cS_2 := \sumflat_{\substack{\b'_2 \sim B / \cN(\b_2^\sharp \r)}}   \lambda_{\b_2^\sharp}^{\dagger\dagger}(\r \b'_2) \cdot \frac{\chi_{\b'_2}(\b_1^\sharp\a')}{\cN(\r\b'_2)^{1/2 -\varepsilon + it}}.
    \end{equation*}
    
    Combining $\b_2^\sharp \a' \in I(S)$ into a single variable and using the divisor bound, we have
    \begin{align*}
        & \sumflat_{\substack{\b_1^\sharp \sim B_1 \\ [\b_1^\sharp] = \cC \cC_1^{-1}}}\sumstar_{\r\sim R} \sumflat_{\substack{\b_2^\sharp \sim B_2 \\ [\b_2^\sharp] = \cC \cC_2^{-1}}} \sum_{\substack{\a' \in I(S) \\ \a' \sim K}} |\cS_1|^2 \\
        \ll_\varepsilon\ & Z \max_{\substack{B_2 K \leq K_1 \leq 2 B_2K \\ \frac{B}{4 B_1 R} \leq L_1 \leq \frac{B}{B_1 R}}} \sL_2(K_1, L_1) \sumflat_{\substack{\b_1^\sharp \sim B_1}} \sumstar_{\r\sim R} \sumstar_{\substack{\b'_1 \sim B / \cN(\b_1^\sharp \r)}} \frac{|\lambda_{\b_1^\sharp\r\b'_1}|^2}{\cN(\r\b'_1)^{1-2\varepsilon}} \\
        \ll_\varepsilon\ & Z^{O(1)} \max_{\substack{B_2 K \leq K_1 \leq 2 B_2K \\ \frac{B}{4 B_1 R} \leq L_1 \leq \frac{B}{B_1 R}}} \frac{B_1}{B} \cdot \sL_2(K_1, L_1).
    \end{align*}
    An analogous inequality holds for $|\cS_2|^2$, combining $\b_1^\sharp \a'$. Thus applying the Cauchy--Schwarz inequality to the sums over $\b_1^\sharp, \b_2^\sharp, \r, \a'$ in \eqref{eq:endgame_Poisson} gives that $\Sigma_3(A, B, \lambda, \Vec{\cC})$ is 
    \begin{align*}
        &\ll_\varepsilon 1 + Z^{O(1)} \sup_{\substack{1\leq B_1, B_2, R < 2 B }} \max_{\substack{\frac{B_1B_2}{B_i} \leq K_i \leq \frac{2 Z B^2}{A B_i } \\ \frac{B}{4 B_i R} \leq L_i \leq \frac{B}{B_i R}}} \frac{A (B_1B_2)^{1/2}}{B}\cdot  \sL_2(K_1, L_1)^{1/2} \cdot \sL_2(K_2, L_2)^{1/2} \\
        & \ll 1 + Z^{O(1)} \sup_{\substack{1\leq B_1, B_2 \leq B }} \max_{\substack{1 \leq K_i \leq \frac{2 Z B^2}{A B_i } \\ 1 \leq L_i \leq \frac{B}{B_i}}} \frac{A (B_1B_2)^{1/2}}{B} \cdot \sL_2(K_1, L_1)^{1/2} \cdot \sL_2(K_2, L_2)^{1/2},
    \end{align*}
    and the desired result follows after adjusting $\varepsilon > 0$.
\end{proof}


\section{Removal of \texorpdfstring{$n$}{}-th powers}

The key improvement in our argument occurs in the passage from $\sL_2$ to $\sL_1$. This is where $n$-th power free ideals are more convenient than squarefree ones. Indeed, to extract a sum over squarefree ideals, it is possible to decompose $\a = \a_1\a_2^2 \cdots \a_{n-1}^{n-1} \a_n^n$ for squarefree $\a_1, \a_2, \dots, \a_{n-1}$ and localize each variable. Then it is unclear how to do better than to trivially bound all except the longest of the sums over $\a_1, \dots, \a_{n-1}$. On the other hand, in our case it is possible to sum over the $n$-th power free ideal $\a_1\a_2^2 \cdots \a_{n-1}^{n-1}$ and trivially bound only the sum over $\a_n$, which is necessarily very short.

\begin{lemma}[Trivial estimation of $n$-th powers]\label{lem:n-th_powers}
    For any $A, B \geq 1$ and $\varepsilon > 0$ we have
    \begin{equation*}
        \sL_2(A, B) \ll_\varepsilon (AB)^\varepsilon \max_{1\leq R \leq (2A)^{1/n}} R \cdot \sL_1\Big(\frac{A}{R^n}, B\Big).
    \end{equation*}
\end{lemma}

\begin{proof}
    Any ideal $\a \in I(S)$ can be uniquely decomposed as $\a = \d\cdot \r^n$ for $\d, \r \in I(S)$ with $\d$ $n$-th power free. Covering the range of $\r$ with dyadic intervals $\r \sim R'$, we have
    \begin{equation*}
         \Sigma_2(A, B, \lambda) = \sum_{\substack{\a \in I(S) \\ \a \sim A}} \Big| \sumflat_{\b \sim B} \lambda_\b \chi_\b(\a) \Big|^2 \ll_\varepsilon A^\varepsilon \max_{\frac{1}{2}\leq R' \leq \frac{(2A)^{1/n}}{2}} \sum_{\substack{\r \in I(S) \\ \r \sim R'}} \sumflat_{\d \sim A / \cN(\r)^n} \Big| \sumflat_{\b \sim B} \lambda_\b \cdot \chi_\b(\r)^n \cdot \chi_\b(\d) \Big|^2.
    \end{equation*}
    We now separate into classes $[\d] = \cC \in R_\c$ and apply reciprocity to replace $\chi_{\b}(\d)$ with $\alpha([\b], \cC) \cdot \chi_\d(\b)$. Thus
    \begin{align*}
        \Sigma_2(A, B, \lambda) &\ll_\varepsilon A^\varepsilon \max_{\substack{\frac{1}{2}\leq R' \leq \frac{(2A)^{1/n}}{2} \\ \cC \in R_\c}} \sum_{\substack{\r \in I(S) \\ \r \sim R'}}  \sL_1\Big(\frac{A}{\cN(\r)^n}, B\Big) \cdot \sumflat_{\b \sim B} |\lambda_\b  \cdot \chi_\b(\r)^n \cdot \alpha([\b], \cC)|^2 \\
        & \ll A^\varepsilon \max_{1\leq R \leq (2A)^{1/n}} R\cdot \sL_1\Big(\frac{A}{R^n}, B\Big) \cdot \normt{\lambda}^2,
    \end{align*}
    so the result follows.
\end{proof}


\section{The recursive argument}

The proof of \cref{thm:main} uses the same iterative scheme as previous works, inserting our improved bounds. For any $\alpha > 1$, let $(E_\alpha)$ denote the statement
\begin{equation*}
    \sL_1(A, B) \ll_\varepsilon (AB)^\varepsilon \Big(A + B^\alpha + A\Big(\frac{B^2}{A}\Big)^{1/n} + B\Big(\frac{A^2}{B}\Big)^{1/n}\Big) \qquad \text{for every $\varepsilon > 0$ and $A, B \geq 1$.}
\end{equation*}
Observe from \cref{lem:symmetry} that $(E_\alpha)$ is equivalent to 
\begin{equation*}
    \sL_1(A, B) \ll_\varepsilon (AB)^\varepsilon \Big(A^\alpha + B + A\Big(\frac{B^2}{A}\Big)^{1/n} + B\Big(\frac{A^2}{B}\Big)^{1/n}\Big) \qquad \text{for every $\varepsilon > 0$ and $A, B \geq 1$.}
\end{equation*}

\begin{prop}[Iteration]\label{prop:iteration}
    Assume that $(E_\alpha)$ holds for some $\alpha > 1+\frac{1}{n}$. Then $(E_{2 - \frac{n-1}{n\alpha - 1}})$ also holds. 
\end{prop}

\begin{proof}
    Let $A, B \geq 1$ and let $\varepsilon > 0$, which may change from one line to the next. Assuming $(E_\alpha)$ for $\alpha > 1+\frac{1}{n}$, observe from \cref{lem:n-th_powers} that
    \begin{align*}
        \sL_2(A, B) & \ll_\varepsilon (AB)^\varepsilon \max_{1\leq R \leq (2A)^{1/n}} R \cdot \sL_1\Big(\frac{A}{R^n}, B\Big) \\
        & \ll_\varepsilon (AB)^\varepsilon \max_{1\leq R \leq (2A)^{1/n}} R \cdot \Big(\Big(\frac{A}{R^n}\Big)^\alpha + B + \frac{A}{R^n} \Big(\frac{B^2R^n}{A}\Big)^{1/n} + B\Big(\frac{A^2}{BR^{2n}}\Big)^{1/n}\Big) \\
        & \ll_\varepsilon (AB)^\varepsilon \Big(A^\alpha + A^{1/n}B + A \Big(\frac{B^2}{A}\Big)^{1/n} + B\Big(\frac{A^2}{B}\Big)^{1/n}\Big) \ll_\varepsilon (AB)^\varepsilon \Big(A^\alpha + A^{1/n}B \Big).
    \end{align*}
    
    Substituting this into \cref{lem:Poisson_K_3}, we obtain
    \begin{align*}
        & \sL_3(A, B)  \ll_\varepsilon A + (AB)^\varepsilon \max_{\substack{1\leq B_1, B_2 \leq B }} \max_{\substack{1 \leq K_i \leq (AB)^\varepsilon\frac{B^2}{A B_i } \\ 1 \leq L_i \leq \frac{B}{B_i}}} \frac{A (B_1B_2)^{1/2}}{B} \cdot \sL_2(K_1, L_1)^{1/2} \cdot \sL_2(K_2, L_2)^{1/2} \\
        & \ll_\varepsilon A + (AB)^\varepsilon \max_{\substack{1\leq B_1, B_2 \leq B }} \max_{\substack{1 \leq K_i \leq (AB)^\varepsilon\frac{B^2}{A B_i } \\ 1 \leq L_i \leq \frac{B}{B_i}}} \frac{A (B_1B_2)^{1/2}}{B} \cdot \Big(K_1^\alpha + K_1^{1/n}L_1 \Big)^{1/2} \cdot \Big(K_2^\alpha + K_2^{1/n}L_2 \Big)^{1/2}.
    \end{align*}
    The product inside the maxima is
    \begin{align*}
        & = \Big[\frac{A^2 B_1B_2}{B^2} \cdot \Big(K_1^\alpha + K_1^{1/n}L_1 \Big) \cdot \Big(K_2^\alpha + K_2^{1/n}L_2 \Big) \Big]^{1/2} \\
        & \ll_\varepsilon(AB)^\varepsilon \Big[\frac{A^2 B_1B_2}{B^2} \cdot \Big(\Big(\frac{B^2}{AB_1}\Big)^\alpha + \Big(\frac{B^2}{AB_1}\Big)^{1/n} \frac{B}{B_1} \Big) \cdot \Big(\Big(\frac{B^2}{AB_2}\Big)^\alpha + \Big(\frac{B^2}{AB_2}\Big)^{1/n} \frac{B}{B_2} \Big) \Big]^{1/2}.
    \end{align*}
    Note that since $\alpha > 1+\frac{1}{n}$, the exponents of $B_1$ and $B_2$ in each of the resulting terms are negative. Hence we may set $B_1 = B_2=1$ and obtain
    \begin{align*}
        \sL_3(A, B) & \ll_\varepsilon A + (AB)^\varepsilon \cdot \frac{A }{B}  \cdot \Big(\Big(\frac{B^2}{A}\Big)^\alpha + \Big(\frac{B^2}{A}\Big)^{1/n} B \Big) \\
        & \ll_\varepsilon (AB)^\varepsilon \Big(A +A^{1-\alpha} B^{2\alpha-1} + A \Big(\frac{B^2}{A}\Big)^{1/n} \Big).
    \end{align*}

    Inserting the inequality above into \cref{lem:K_1_K_2_inequality} and \cref{lem:GCD_sorting} leads, for some $1 \leq \Delta_1 \ll \Delta_2$, to
    \begin{align}
        \sL_1(A, B) \ll \sL_2(A, B) & \ll_\varepsilon (AB)^\varepsilon \cdot \sL_3\Big( \frac{A}{\Delta_1}, \frac{B}{\Delta_2}\Big) \nonumber \\
        & \ll_\varepsilon (AB)^\varepsilon \Big(\frac{A}{\Delta_1} +\frac{A^{1-\alpha} B^{2\alpha-1}}{(\Delta_2/\Delta_1)^{\alpha-1} \Delta_2^{\alpha} } + A \Big(\frac{B^2}{A}\Big)^{1/n} \frac{1}{\Delta_1^{1-1/n}\Delta_2^{2/n}} \Big) \nonumber \\
        & \ll_\varepsilon (AB)^\varepsilon \Big(A +A^{1-\alpha} B^{2\alpha-1} + A \Big(\frac{B^2}{A}\Big)^{1/n} \Big). \label{eq:recursive_bound}
    \end{align}

    If $A \geq B^{2-\frac{n}{n\alpha - 1}}$, then the second term in \eqref{eq:recursive_bound} is bounded by the third, so we obtain the desired bound. Otherwise, if $A \leq B^{2-\frac{n}{n\alpha - 1}}$, then we apply \cref{lem:monotonicity} followed by \eqref{eq:recursive_bound} to bound
    \begin{align*}
        \sL_1(A, B) & \ll \sL_1\big(CB^{2-(\alpha - \frac{1}{n})^{-1}} \log(2AB), B\big) \ll_\varepsilon (AB)^\varepsilon \Big(B^{2-\frac{n}{n\alpha - 1}} + B^{2-\frac{n-1}{n\alpha - 1}} + B^{2-\frac{n-1}{n\alpha - 1}} \Big) \\
        & \ll (AB)^\varepsilon B^{2-\frac{n-1}{n\alpha - 1}} \ll (AB)^\varepsilon \Big(A +B^{2-\frac{n-1}{n\alpha - 1}} + A \Big(\frac{B^2}{A}\Big)^{1/n} \Big). 
    \end{align*}
    This establishes $(E_{2 - \frac{n-1}{n\alpha - 1}})$, as desired.
\end{proof}

For completeness, let us establish $(E_\alpha)$ for some $\alpha > 1 + \frac{1}{n}$.

\begin{lemma}[Initial bound]\label{lem:E_2}
    The statement $(E_2)$ holds.
\end{lemma}

\begin{proof}
    Let $A, B \geq 1$ and $\varepsilon > 0$. By Cauchy--Schwarz, we have $\sL_2(A, B) \ll AB$. Inserting this into \cref{lem:Poisson_K_3} gives
    \begin{equation*}
        \sL_3(A, B) \ll_\varepsilon (AB)^\varepsilon (A + B^2). 
    \end{equation*}
    The same bound follows for $\sL_2(A, B)$ by \cref{lem:GCD_sorting}, and then for $\sL_1(A, B)$ by \cref{lem:K_1_K_2_inequality}.
\end{proof}

We are ready to prove our main theorem.

\begin{proof}[Proof of \cref{thm:main}]
    We start with $(E_2)$ from \cref{lem:E_2}, and iterate the bound from \cref{prop:iteration}. Since for $\alpha > 1 + \frac{1}{n}$ we have $2 - \frac{n-1}{n\alpha - 1}< \alpha$, note that $\alpha$ converges to the fixed point $1 + \frac{1}{n}$, hence we obtain the bound
    \begin{equation*}
        \sL_1(A, B) \ll_\varepsilon (AB)^\varepsilon \Big(A + B^{1+1/n} + A\Big(\frac{B^2}{A}\Big)^{1/n} + B\Big(\frac{A^2}{B}\Big)^{1/n}\Big).
    \end{equation*}
    
    If $A \geq B$, then the third term on the RHS dominates the second one and the latter can be removed. Otherwise, we apply the bound above to $\sL_1(B, A)$, which by \cref{lem:symmetry} is comparable to $\sL_1(A, B)$, and the second term $A^{1+1/n}$ is again dominated by the third. Either way, the desired bound (corresponding to $\alpha = 1$) for $\sL_1(A, B)$ is obtained. 
    
    This implies the same bound for $\Theta_n(A, B)$, after splitting both sums into dyadic ranges and applying Cauchy--Schwarz. Therefore \cref{thm:main} follows.
\end{proof}


\section{Cubic moments in the conductor aspect}

To prove \cref{thm:cubic_second_moment}, we start with a general large sieve in conductor aspect. It can probably be improved in certain ranges, but suffices for our purposes.

\begin{prop}[Bound for cubic sums in terms of conductor]\label{prop:cubic_conductor_bound}
    Assume that $n = 3$. Then for any $A, B, Q \geq 1$, denoting $M = \min(A, B)$ and $N = \max(A, B)$, we have
    \begin{equation}\label{eq:cubic_conductor_bound}
        \sup_{\normt{\lambda} = 1}\sum_{\substack{\a \in I(S) \\ \a \sim A}} \sum_{\substack{\b \in I(S) \\ \b \sim B}} \Big| \sumstar_{\substack{\q \sim Q}} \lambda_\q \chi_{\q}(\a \b^2) \Big|^2 \ll_{\varepsilon} (ABQ)^\varepsilon ( AB +  M^{2/3} N^{1/3} Q + M^{1/3} N^{2/3} Q^{2/3} ).
    \end{equation}
\end{prop}

\begin{proof}
    Since $n=3$, we have $\chi_\q(\a\b^2) = \chi_\q(\a) \overline{\chi_\q(\b)}$. Hence both sides of \eqref{eq:cubic_conductor_bound} are symmetric in $A$ and $B$, so we may assume without loss of generality that $A \leq B$. 

    Inserting the smooth non-negative weight $W(\cN(\b)/B)$, opening the square, and sorting by GCD as in \eqref{eq:K_3_intermediate}, the sum over $\b$ on the LHS of \eqref{eq:cubic_conductor_bound} is
    \begin{equation*}
        \ll \max_{\cC \in R_\c} \sumstar_{\d} \Bigg|  \sumstar_{\substack{\q_1, \q_2 \sim Q/\cN(\d) \\ (\q_1, \q_2) = (\q_1 \q_2, \d) =1 \\ [\q_1] = [\q_2] = \cC [\d]^{-1}}} \lambda_{\q_1 \d} \overline{\lambda_{\q_2 \d}} \cdot \chi_{\q_1\d}(\a) \overline{\chi_{\q_2\d}(\a)} \sum_{\substack{\b \in I(S)}} W\Big(\frac{\cN(\b)}{B}\Big) \overline{\chi_{\q_1\d}(\b)} \chi_{\q_2\d}(\b) \Bigg| .
    \end{equation*}
    Note, as in the paragraphs following \eqref{eq:K_3_intermediate}, that $\overline{\chi_{\q_1\d}(\b)} \chi_{\q_2\d}(\b) = \overline{\tchi_{\q_1, \q_2}(\b)} \bbone_{(\b, \d)=1}$. Removing the condition $(\b, \c\d)=1$ by M\"obius inversion and recalling the decomposition $\tchi_{\q_1, \q_2} = \chi^{(\q_1)} \cdot \overline{\chi^{(\q_2)}}$ of \eqref{eq:_primitive_chi_decomposition}, the display above is
    \begin{equation}\label{eq:cubic_interm}
        \ll \sumstar_{(\d, \a)=1} \max_{\cC \in R_\c} \Bigg| \sum_{\r \mid \c\d} \mu(\r)  \sumstar_{\substack{\q_1, \q_2 \sim Q/\cN(\d) \\ (\q_1, \q_2) = 1 \\ [\q_1] = [\q_2] = \cC}} \lambda^\dagger_{\q_1} \overline{\lambda^\dagger_{\q_2}} \cdot \chi_{\q_1}(\a) \overline{\chi_{\q_2}(\a)}\sum_{\substack{\b \subseteq \cO}} W\Big(\frac{\cN(\b)}{B /\cN(\r)}\Big)  \overline{\tchi_{\q_1, \q_2}(\b)} \Bigg|
    \end{equation}
    for
    \begin{equation*}
        \lambda_\q^\dagger = \lambda_\q^\dagger(\d, \r) := \lambda_{\q\d} \cdot \overline{\chi^{(\q)}(\r)} \cdot \bbone_{(\q, \d)=1}.
    \end{equation*}

    We split the summation over $\d$ into dyadic ranges $\d \sim D$, where $1\leq D < 2Q$ runs over powers of $2$. The contribution of $Q/2 \leq D < 2Q$ to \eqref{eq:cubic_interm} can be bounded trivially by $\ll \sum_{\d \sim D}\sum_{\cN(\q) \ll 1} \sum_{\r \mid \c\d} |\lambda_{\q\d}|^2 \cdot \frac{B}{\cN(\r)} \ll_\varepsilon Q^\varepsilon B$, leading to the term $AB$ in \eqref{eq:cubic_conductor_bound} once we sum over $\a$.

    For the remaining ranges $1 \leq D < Q/2$, note that $\overline{\tchi_{\q_1, \q_2}}$ is a non-trivial primitive ray class character of conductor $\q_1\q_2$. Thus we may apply the Poisson summation formula of \cref{lem:general_Poisson_formula}, combined with \cref{lem:root_number} (whose result is simplified since $\q_1$ and $\q_2$ are squarefree) to \eqref{eq:cubic_interm}. Denoting $\widecheck{\lambda}_\q := \lambda^\dagger_\q \cdot \overline{\beta_\q}$, we obtain
    \begin{equation*}
        \ll \sumstar_{\substack{\d\sim D }} \sum_{\r \mid \c\d} \max_{\cC \in R_\c} \Bigg| \sumstar_{\substack{\q_1, \q_2 \sim Q/\cN(\d) \\ (\q_1, \q_2) = 1 \\ [\q_1] = [\q_2] = \cC}} \widecheck{\lambda}_{\q_1} \overline{\widecheck{\lambda}_{\q_2}} \cdot \chi_{\q_1}(\a) \overline{\chi_{\q_2}(\a)}\cdot  \frac{B}{ \sqrt{\cN(\q_1 \q_2 \r^2)}} \sum_{\b \subseteq \cO} \dot{W}\Big(\frac{\cN(\b)}{\cN(\q_1 \q_2 \r) / B}\Big) \tchi_{\q_1, \q_2}(\b) \Bigg|.
    \end{equation*}
    
    Since $\tchi_{\q_1, \q_2}$ is non-trivial, note that $\b = (0)$ does not contribute, so we may (uniquely) decompose $\b = \s\e$ for $\s \mid \c^\infty$ and $\e \in I(S)$. Then $\tchi_{\q_1, \q_2}(\b) = \chi^{(\q_1)}(\s)\overline{\chi^{(\q_2)}(\s)} \cdot \chi_{\q_1}(\e)\overline{\chi_{\q_2}(\e)}$, so the display above is
    \begin{equation*}
        \ll \sumstar_{\substack{\d\sim D }} \sum_{\r \mid \c\d} \frac{B}{\cN(\r)} \sum_{\s \mid \c^\infty} \max_{\cC \in R_\c} \sum_{\e \in I(S)} \Bigg| \sumstar_{\substack{\q_1, \q_2 \sim Q/\cN(\d) \\ (\q_1, \q_2) = 1 \\ [\q_1] = [\q_2] = \cC}} \frac{\lambda'_{\q_1} \overline{\lambda'_{\q_2}}}{\cN(\q_1 \q_2)^{1/2}} \cdot \chi_{\q_1}(\a\e) \overline{\chi_{\q_2}(\a\e)} \cdot \dot{W}\Big(\frac{\cN(\s\e)}{\cN(\q_1 \q_2 \r) / B}\Big)  \Bigg|,
    \end{equation*}
    for $\lambda'_\q := \widecheck{\lambda}_\q \cdot \chi^{(\q)}(\s)$. Split into dyadic ranges $\e \sim E$, where $E\geq 1$ runs over powers of $2$. Denote $Z := (ABQ)^\varepsilon$, where $\varepsilon > 0$. Observe by the decay of $\dot{W}$ in \eqref{eq:W_dot_bound} that the values $E > Z Q^2 \cN(\r) / B D^2 \cN(\s)$ contribute an error $\ll_\varepsilon (ABQ)^{-100}$.

     To separate variables, insert the definition \eqref{eq:W_dot_def} of $\dot{W}$ and shift the line of integration to $\Re(s) = \varepsilon$. The rapid decay of the Mellin transform $\widehat{W}$ allows us to truncate to $|\Im(s)| \leq Z$, once again at the cost of an overall error $\ll_\varepsilon (ABQ)^{-100}$. Thus apart from this error term, the sum over $\e$ in this range is
     \begin{align*}
         &\ll_\varepsilon Z^{O(1)} \int_{-Z}^Z \sum_{\substack{\e \in I(S) \\ \e \sim E}} \Bigg| \sumstar_{\substack{\q_1, \q_2 \sim Q/\cN(\d) \\ (\q_1, \q_2) = 1 \\ [\q_1] = [\q_2] = \cC}} \frac{\lambda'_{\q_1} \overline{\lambda'_{\q_2}}}{\cN(\q_1 \q_2)^{1/2-\varepsilon + it}} \cdot \chi_{\q_1}(\a\e) \overline{\chi_{\q_2}(\a\e)} \Bigg| \, |dt| \\
         & \ll_\varepsilon Z^{O(1)} \int_{-Z}^Z \sumstar_\g \frac{1}{\cN(\g)} \sum_{\substack{\e \in I(S) \\ \e \sim E}} \Bigg| \sumstar_{\substack{\f_1, \f_2 \sim Q/\cN(\d\g) \\ [\f_1] = [\f_2] = \cC [\g]^{-1}}} \frac{\lambda'_{\f_1\g} \overline{\lambda'_{\f_2\g}}}{\cN(\f_1 \f_2)^{1/2-\varepsilon + it}} \cdot \chi_{\f_1}(\a\e) \overline{\chi_{\f_2}(\a\e)} \Bigg| \, |dt|,
     \end{align*}
     where the last step is by M\"obius inversion and we dropped the condition $(\g, \a\e)=1$ by positivity. 
     
     At this point we re-introduce the sum over $\a \in I(S)$. Note that the coefficients $\lambda'_\q$ do not depend on $\a$ or $\e$ (but only on $\q, \d, \r, \s$). So we may combine these variables into $\u := \a\e \in I(S)$, which satisfies $AE \leq \cN(\u) < 4AE$, and use a divisor bound. Furthermore, we have separated the variables $\f_1$ and $\f_2$, so we may apply Cauchy--Schwarz to the sum over $\u$ and obtain (two copies of) the kernel $\sL_2$. More precisely, we conclude that the contribution of these localized ranges to \eqref{eq:cubic_conductor_bound} is
    \begin{equation*}
        \ll_\varepsilon Z^{O(1)} B \sumstar_{\d \sim D} \sum_{\r \mid \c\d} \sum_{\substack{\s \mid \c^\infty \\ \cN(\s) \ll ZQ^2}}\sumstar_{\g} \Big[ \sL_2\Big( AE, \frac{Q}{\cN(\d\g)} \Big) + \sL_2\Big( 2AE, \frac{Q}{\cN(\d\g)} \Big) \Big] \sumstar_{\substack{\f \sim Q/\cN(\d\g)}} \frac{|\lambda'_{\g\f}|^2}{\cN(\r\f\g)} .
     \end{equation*}

     Further localize $\r \sim R$ into dyadic intervals. Recall that $E \leq \frac{Z Q^2 \cN(\r)}{B D^2 \cN(\s)} \ll \frac{ZQ^2 R}{BD^2}$, where $1 \ll R \ll D \ll Q$. Note from \cref{lem:n-th_powers} and \cref{thm:main} that we have
     \begin{equation}\label{eq:L_2_consequence}
        \sL_2(X, Y) \ll_\varepsilon (XY)^\varepsilon \big(X + X^{1/3}Y + (XY)^{2/3}\big)
     \end{equation} 
     for any $X, Y > 0$. Inserting this bound into the display above, we obtain
     \begin{align*}
         & \ll_\varepsilon Z^{O(1)} B \sumstar_{\d \sim D} \sum_{\substack{\r \sim R \\ \r \mid \c\d}} \Big[ AE + (AE)^{1/3} \frac{Q}{D} + \Big(\frac{AEQ}{D}\Big)^{2/3} \Big] \sum_{\substack{\frakh \sim Q/\cN(\d)}} \frac{|\widecheck{\lambda}_{\frakh}|^2}{\cN(\r\frakh)} \\
         & \ll_\varepsilon Z^{O(1)} \frac{BD}{QR} \Big[ \frac{AQ^2R}{BD^2} + \Big(\frac{AQ^2R}{BD^2}\Big)^{1/3}\frac{Q}{D} + \Big(\frac{AQ^3R}{BD^3}\Big)^{2/3} \Big] \sumstar_{\d \sim D} \sum_{\substack{\frakh \sim Q/\cN(\d)}} |\lambda_{\frakh\d}|^2 \\
         & \ll_\varepsilon Z^{O(1)} \Big[ \frac{AQ}{D} + A^{1/3}\Big(\frac{BQ}{DR}\Big)^{2/3} + A^{2/3}\frac{Q}{D}\Big(\frac{B}{R}\Big)^{1/3} \Big] \\
         &\ll_\varepsilon Z^{O(1)} \big[ AQ + A^{1/3}B^{2/3} Q^{2/3} + A^{2/3}B^{1/3} Q \big] \\
         & \ll_\varepsilon Z^{O(1)} \big[A^{1/3}B^{2/3} Q^{2/3} + A^{2/3}B^{1/3} Q \big],
        \end{align*}
        since we assumed that $A\leq B$. This gives the desired result, after incorporating the previous term $\ll_\varepsilon ZAB$ corresponding to the diagonal range $Q/2 \leq D < 2Q$.
\end{proof}

\begin{remark}
    The proof of \cref{prop:cubic_conductor_bound} does not require the full force of \cref{thm:main}, since \eqref{eq:L_2_consequence} can be replaced with the same bound for the kernel where the inner sum is over squarefree ideals, rather than cube-free. Such an estimate follows directly from Heath-Brown's version of the cubic large sieve, as in \cite{DDDS24}*{Proposition 3.13}. However, our \cref{thm:main} would also give a bound similar to \cref{prop:cubic_conductor_bound} when $\q$ is only required to be cube-free. This would lead to the previous complications with root numbers, so we refrain from this extra work since the squarefree version suffices for our purposes.
\end{remark}

The power of \cref{prop:cubic_conductor_bound} comes from the fact that if $Q \leq (AB)^{1/2+\varepsilon}$, then it gives the fully orthogonal bound $\ll_\varepsilon (AB)^{1+2\varepsilon}$. This range of $Q$ is the relevant one for the approximate functional equation of $L(\tfrac{1}{2}, \chi_{\a\b^2})$, so we obtain a tight bound for the corresponding second moment.

\begin{proof}[Proof of \cref{thm:cubic_second_moment}]
    We fix the unitary Hecke character $\psi$, and may assume that $AB$ is sufficiently large in terms of $\psi$ so that $\psi \chi_{\a\b^2}$ has non-trivial conductor. Applying the standard approximate functional equation and Mellin inversion as in \cite{DDDS24}*{(4.10)} and \cite{DIPP26}*{(55)}, for $\a, \b \in I(S)$ with $\mu^2(\a\b)=1$ we have, denoting $Z = (AB)^\varepsilon$, that
    \begin{equation*}
        \sumstar_{\a \sim A} \sumstar_{\substack{\b \sim B \\ (\a, \b)=1}} |L(\tfrac{1}{2}, \psi\chi_{\a\b^2})|^2 \ll_\varepsilon 1 + Z^{O(1)} \int_{-Z}^Z \sumstar_{\a \sim A} \sumstar_{\substack{\b \sim B \\ (\a, \b)=1}} \Big| \sum_{\substack{0 \neq \frakn \subseteq \cO \\ \cN(\frakn) \leq Z(AB)^{1/2}}} \frac{\psi(\frakn) \chi_{\a\b^2}(\frakn)}{\cN(\frakn)^{1/2 + \varepsilon + it}} \Big|^2 \, |dt|.
    \end{equation*}
    Decompose $\frakn = \s\q\r^2$ for $\s \mid \c^\infty$ and $\q, \r \in I(S)$ with $\mu^2(\q)=1$. By Cauchy--Schwarz and convergence of $\sum_{\s \mid \c^\infty} \sum_{\r \in I(S)} \cN(\s\r^2)^{-1/2-\varepsilon}$, the display above is
    \begin{equation}\label{eq:approx_functional_eq}
        \ll_\varepsilon 1 + Z^{O(1)} \int_{-Z}^Z \sum_{\s \mid \c^\infty} \sum_{\r \in I(S)} \cN(\s\r^2)^{-1/2-\varepsilon} \sumstar_{\a \sim A} \sumstar_{\substack{\b \sim B \\ (\a, \b)=1}} \Big| \sumstar_{\substack{\cN(\s\r^2\q) \leq Z (AB)^{1/2}}} \frac{\psi(\q) \chi_{\a\b^2}(\q)}{\cN(\q)^{1/2 + \varepsilon + it}} \Big|^2 \, |dt|.
    \end{equation}
    Splitting the sum over $\q$ into dyadic intervals and applying \cref{prop:cubic_conductor_bound} with $\lambda_\q = \psi(\q) \cN(\q)^{-(1/2+\varepsilon+it)}$, the sums over $\a$ and $\b$ are $\ll_\varepsilon Z^{O(1)} AB$. We obtain the desired bound
    \begin{equation}\label{eq:cubic_first_res}
        \sumstar_{\a \sim A} \sumstar_{\substack{\b \sim B \\ (\a, \b)=1}} |L(\tfrac{1}{2}, \psi\chi_{\a\b^2})|^2 \ll_\varepsilon 1 + Z^{O(1)} AB \sum_{\s \mid \c^\infty} \sum_{\r \in I(S)} \cN(\s\r^2)^{-1/2-\varepsilon} \ll_\varepsilon Z^{O(1)}AB.
    \end{equation}

    Finally, if $\g \in I(S)$ is cube-free, then writing $\g = \a\b^2$ for $\a, \b \in I(S)$ with $\mu^2(\a\b)=1$, we have $\cond(\chi_\g) = \c_\g \a\b$ for some $\c_\g \mid \c$. Thus 
    \begin{equation*}
        \sumflat_{\substack{\cN(\cond(\chi_\g)) \leq X}} |L(\tfrac{1}{2}, \chi_\g)|^2 \leq \sumstar_\a \sumstar_{\substack{\cN(\b) \leq X/\cN(\a)}} \mu^2(\a\b) \cdot |L(\tfrac{1}{2}, \chi_{\a\b^2})|^2 \ll_\varepsilon X^{1+\varepsilon}
    \end{equation*}
    due to \eqref{eq:cubic_first_res}, after localizing the sums over $\a$ and $\b$ in dyadic ranges and taking $\psi=1$. 
\end{proof}


\section{Consequences for moments and zero density}

\begin{proof}[Proof of \cref{thm:general_second_moment}]
    As in the proof of \cite{DIPP26}*{Theorem 10.5.2}, it suffices to show that $Z_1^S(\tfrac{1}{2}, \tfrac{1}{2}, w)$ has meromorphic continuation to $\sigma := \Re(w) > \sigma_n$, with at most double poles at $w = 1$ and $w = \tfrac{1}{2}+\tfrac{1}{n}$, and satisfies a polynomial bound in $w$ of the form
    \begin{equation*}
        (w-1)^2 (w-\tfrac{1}{2} - \tfrac{1}{n})^2 Z_1^S(\tfrac{1}{2}, \tfrac{1}{2}, w) \ll_\sigma (1+|w|)^{A_n}
    \end{equation*}
    in that region, for some fixed $A_n \geq 0$.
    
    Combining equation (18), Propositions 6.3.2 and 7.2.2, and the bounds in \S 8.1.1 of \cite{DIPP26}, it suffices as in \cite{DIPP26}*{Proposition 8.2.2} to prove convergence of
    \begin{equation}\label{eq:S_sum_def}
        S_n(\sigma) := \max_{\psi \in \widehat{R_\c}}\sumstar_{\h} \cN(\h)^{(n-1)(1-\sigma) -2\sigma} \sum_{\f \mid \h^{n-1}} \sumflat_{(\a, \f)=1} \frac{|L(\tfrac{1}{2}, \psi\chi_{\a\f})|^{2}}{\cN(\a)^{1+\varepsilon}}
    \end{equation}
    for $\sigma > \sigma_n$ and some $\varepsilon > 0$. Since $\sigma_n > \frac{n}{n+1}$, we can exchange the order of the sums over $\h$ and $\f$ to obtain
    \begin{equation*}
        S_n(\sigma) \ll_\varepsilon \max_{\psi \in \widehat{R_\c}} \sumflat_{\f} \sumflat_{(\a, \f)=1} \frac{|L(\tfrac{1}{2}, \psi\chi_{\a\f})|^{2}}{\cN(\a)^{1+\varepsilon} \cN(\f^\star)^{\alpha}} 
    \end{equation*}
    for $\alpha := (n+1)\sigma - (n-1) > 1$. Writing $\u := \a\f$, which is $n$-th power free, and factoring $\u = \q_1\q_2^2 \cdots \q_{n-1}^{n-1}$ for $\q_i\in I(S)$ with $\mu^2(\q_1\q_2\cdots \q_{n-1})=1$, the divisor bound gives
    \begin{equation*}
        S_n(\sigma) \ll_\varepsilon \max_{\psi \in \widehat{R_\c}} \sumstar_{\q_1}\sumstar_{\q_2} \cdots \sumstar_{\q_{n-1}} \mu^2(\q_1\q_2\cdots\q_{n-1}) \frac{|L(\tfrac{1}{2}, \psi\chi_{\prod_{i=1}^{n-1} \q_i^i})|^{2}}{\cN(\q_1)^{1+\varepsilon} \cN(\q_2\cdots \q_{n-1})^{\alpha}}.
    \end{equation*}

    To obtain convergence of the sum above when $\sigma > \sigma_n$, it suffices to show that
    \begin{equation*}
        M_n(\Vec{Q}) := \max_{\psi \in \widehat{R_\c}} \sumstar_{\q_1\sim Q_1} \cdots \sumstar_{\q_{n-1} \sim Q_{n-1}} \mu^2(\q_1\cdots\q_{n-1}) \cdot \big|L\big(\tfrac{1}{2}, \psi\chi_{\prod_{i=1}^{n-1} \q_i^i}\big)\big|^{2}
    \end{equation*}
    satisfies, for every $\varepsilon > 0$, the bound
    \begin{equation*}
        M_n(\Vec{Q}) \ll_\varepsilon Q_1^{1+\varepsilon} (Q_2 \cdots Q_{n-1})^{\alpha_n+\varepsilon} \qquad \text{ for  } \qquad \alpha_n := \begin{cases}
            \frac{17}{13} &  \text{ if } n=4, \\
            \frac{3n-7}{2(n-2)} &  \text{ if } n \geq 5.
        \end{cases}
    \end{equation*}
    Denote $Q := \prod_{i=1}^{n-1} Q_i$ and $\p_j := \q_j^{-j} \prod_{i=1}^{n-1} \q_i^i$ for any $1\leq j\leq n-1$. Applying the approximate functional equation as in \eqref{eq:approx_functional_eq}, and trivially bounding the sums over $\q_i$ for $i\neq j$, we have
    \begin{equation*}
        M_n(\Vec{Q}) \ll_\varepsilon Q^\varepsilon \max_{\psi \in \widehat{R_\c}} \sup_{1\leq B\leq Q^{1/2+\varepsilon}} \sup_{t \in \R} \sup_{\substack{\p_j \in I(S)}} \frac{Q}{Q_j} \sumstar_{\substack{\q_j \sim Q_{j} \\ (\q_j, \p_j)=1}} \Big| \sumstar_{\cN(\b) \leq B} \frac{\psi(\b) \chi_{\p_j}(\b) \chi^j_{\q_j}(\b)}{\cN(\b)^{1/2+\varepsilon+it}} \Big|^2.
    \end{equation*}
    Note that $\{\chi_\a^j\}_{\a\in I(S)}$ is a family of $m$-th order Hecke characters for $m = \frac{n}{(n, j)}$, so\footnote{We do not need our stronger result in \cref{thm:main} when $n\geq 5$.}
    \begin{equation*}
        M_n(\Vec{Q}) \ll_\varepsilon \frac{Q^{1+\varepsilon}}{Q_j} \sup_{1\leq B\leq Q^{1/2+\varepsilon}} \Xi_m(Q_j, B) \ll_\varepsilon Q^{1+\varepsilon} \Big( 1 + \frac{Q^{1/2}}{Q_j} + \Big(\frac{Q}{Q_j}\Big)^{1/3} \Big),
    \end{equation*}
    by \eqref{eq:BGL_bound} for $m \geq 3$ and \eqref{eq:quadratic_bound} for $m=2$. 
    
    Write $R := \prod_{i=2}^{n-1} Q_i$ and choose $j$ with $Q_j = \max_{i} Q_i$. Note that $Q_1 \leq Q_j$, so $\frac{Q}{Q_{j}}\leq \frac{Q}{Q_1} = R$. Also $R \leq Q_j^{n-2}$, hence $\frac{Q}{Q_j^2} \leq \frac{R}{Q_j} \leq R^{\frac{n-3}{n-2}}$. For $n\geq 5$, we conclude that
    \begin{align*}
        M_n(\Vec{Q}) \ll_\varepsilon Q^{1+\varepsilon} \Big( 1 + R^{\frac{n-3}{2(n-2)}} + R^{1/3}  \Big) \ll_\varepsilon Q_1^{1+\varepsilon} R^{\frac{3n-7}{2(n-2)} + \varepsilon},
    \end{align*}
    as claimed. 
    
    If instead $n = 4$, then we use the sharper bound
    \begin{align*}
        M_4(\Vec{Q}) &\ll_\varepsilon Q^{1+\varepsilon} \min\Big(1 +\Big(\frac{Q_1Q_3}{Q_2}\Big)^{1/2}, 1 +\Big(\frac{Q_2Q_3}{Q_1}\Big)^{1/2} + (Q_2Q_3)^{1/4} + \frac{(Q_2Q_3)^{3/8}}{Q_1^{1/8}} \Big) \\
        & = Q^{1+\varepsilon} \min\Big(1 +\frac{(Q_1 R)^{1/2}}{Q_2}, 1 +\Big(\frac{R}{Q_1}\Big)^{1/2} + R^{1/4} + \frac{R^{3/8}}{Q_1^{1/8}} \Big)
    \end{align*}
    obtained using \eqref{eq:quadratic_bound} for $j=2$ and \cref{thm:main} for $j=1$, where we now make no assumptions on the size of the corresponding $Q_j$. 
    
    Since we aim for an overall bound $M_4(\Vec{Q})\ll_\varepsilon Q^{1+\varepsilon} R^{\alpha_4-1}$, note (using $\alpha_4 = \frac{17}{13} > \frac{5}{4}$) that the first bound is enough unless
    \begin{equation}\label{eq:first_ineq}
        \frac{Q_2}{Q_1^{1/2}} \leq R^{3/2-\alpha_4},
    \end{equation}
    and the second bound is enough unless
    \begin{equation}\label{eq:second_ineq}
        Q_1 \leq R^{\max(3-2\alpha_4, 11-8\alpha_4)} = R^{11-8\alpha_4}.
    \end{equation}
    If both \eqref{eq:first_ineq} and \eqref{eq:second_ineq} hold, then $Q_2 \leq R^{7-5\alpha_4}$ and therefore $Q_3 = \frac{R}{Q_2} \geq R^{5\alpha_4 - 6}$. But then we may take $j=3$ at the outset and use \cref{thm:main} to obtain
    \begin{align*}
        M_4(\Vec{Q}) &\ll_\varepsilon Q^{1+\varepsilon} \Big( 1 +\Big(\frac{Q_1Q_2}{Q_3}\Big)^{1/2} + (Q_1Q_2)^{1/4} + \frac{(Q_1Q_2)^{3/8}}{Q_3^{1/8}} \Big) \\
        &\leq Q^{1+\varepsilon} \Big( 1 + \big(R^{24-18\alpha_4}\big)^{1/2} + \big(R^{18-13\alpha_4}\big)^{1/4} + \big(R^{60-44\alpha_4}\big)^{1/8} \Big).
    \end{align*}
    
    In the first three terms inside the parenthesis above, the exponent of $R$ is $< \alpha_4-1$ for $\alpha_4 = \frac{17}{13}$. The exponent of the fourth term equals $\alpha_4-1$ with this choice (so the fourth term is the limiting one), finishing the proof.
\end{proof}

\begin{proof}[Proof of \cref{thm:cubic_secondary_term}]
    As in the general case \eqref{eq:S_sum_def}, it suffices to prove convergence of
    \begin{equation*}
        S_3(\sigma) := \max_{\psi \in \widehat{R_\c}}\sumstar_{\h} \cN(\h)^{2-4\sigma} \sum_{\f \mid \h^2} \sumflat_{(\a, \f)=1} \frac{|L(\tfrac{1}{2}, \psi\chi_{\a\f})|^2}{\cN(\a)^{1+\varepsilon}}
    \end{equation*}
    for $\sigma > \frac{3}{4}$ and some $\varepsilon > 0$. Write $\b := \a\f$, which is cube free. Note that $\b^\star = \a^\star \f^\star \mid \a\h$, hence $\cN(\cond(\chi_\b)) \ll \cN(\a\h)$. Writing $\h = \f^\star \g$, note that for $\sigma \geq \frac{3}{4}+\varepsilon$ we have
    \begin{align*}
        S_3(\sigma) &= \max_{\psi \in \widehat{R_\c}}\sumflat_{\f} \cN(\f^\star)^{2-4\sigma}  \Big( \sumstar_{(\g, \f)=1} \cN(\g)^{2-4\sigma}\Big)\sumflat_{(\a, \f)=1} \frac{|L(\tfrac{1}{2}, \psi\chi_{\a\f})|^2}{\cN(\a)^{1+\varepsilon}} \\
        & \ll_\varepsilon \max_{\psi \in \widehat{R_\c}}\sumflat_{\f} \sumflat_{(\a, \f)=1} \frac{|L(\tfrac{1}{2}, \psi\chi_{\a\f})|^2}{\cN(\a\f^\star)^{1+\varepsilon}} \ll_\varepsilon \max_{\psi \in \widehat{R_\c}}\sumflat_{\g} \frac{|L(\tfrac{1}{2}, \psi\chi_{\g})|^2}{\cN(\g^\star)^{1+\frac{\varepsilon}{2}}}.
    \end{align*}
    The sum over $\g$ is $\ll_\varepsilon 1$ due to \cref{thm:cubic_second_moment}, so this finishes the proof.
\end{proof}

\begin{proof}[Proof of \cref{thm:zero_density}]
    We closely follow the proof of \cite{BGL14}*{Corollary 1.6}. The only difference is that we replace the squarefree restriction on $\a$ with $n$-th power free, and in the steps below \cite{BGL14}*{(5.3)} we replace the applications of \cite{BGL14}*{Theorem 1.3} with our \cref{thm:main}. For any $1 \leq X \leq Y$, this leads to
    \begin{align*}
        \sumflat_{\cN(\a) \leq N} N(\sigma, T, \a) \ll_\varepsilon (NTY)^\varepsilon T \Big( Y^{1/2-\sigma} \big(NT^{d/2}\big)^{1/2} \big(X + N + X^{2/n}N^{1-1/n} + X^{1-1/n}N^{2/n}\big)^{1/2} \\
        + Y^{2-2\sigma} + NX^{1-2\sigma} + \sup_{X \leq U \leq Y} U^{1-2\sigma} \big(N^{2/n} U^{1-1/n} + N^{1-1/n} U^{2/n} \big) \Big).
    \end{align*}
    
    The supremum over $U$ in the display above splits the range of $\sigma$ into three parts, depending on its position relative to $\tfrac{1}{2}+\frac{1}{n}$ and $1 - \frac{1}{2n}$. For $n=3$, the middle range collapses as in \cite{BGL14}. We will choose
    \begin{equation*}
        X = N^x \qquad \text{ and } \qquad Y = N^y T^{z}
    \end{equation*}
    for some $0 \leq x \leq y$ and $z \geq 0$. We aim for a final bound of the form $\ll_\varepsilon (NT)^\varepsilon N^{g_n(\sigma)} T^{h_n(\sigma)}$. Note that $h_n(\sigma) \leq 1 + \max((\frac{1}{2}-\sigma)z + \frac{d}{4}, (2-2\sigma)z)$. Making the two terms in the maximum match, we choose $z = \frac{d}{6-4\sigma}$, so that $h_n(\sigma) \leq 1+\frac{d(1-\sigma)}{3-2\sigma}$ for all $\frac{1}{2}\leq \sigma \leq 1$. 
    
    Let us now compute the exponent of $N$. Denote $q := 2\sigma - 1$. If $\frac{1}{2}\leq \sigma \leq \frac{1}{2}+\frac{1}{n}$, so that $0 \leq nq \leq 2$, then choose $x = \frac{1}{2}$ and $y = \frac{2}{4-nq} \in[\frac{1}{2}, 1]$. This simple choice of $x$ and the restriction on $\sigma$ give the final bound $g_n(\sigma) \leq \max(1-\frac{qy}{2}, 1-\frac{1}{n} + (\frac{2}{n}-q)y) = 1 - \frac{q}{4-nq}$.

    If $\frac{1}{2}+\frac{1}{n} \leq \sigma \leq 1-\frac{1}{2n}$, so that $2 \leq nq \leq n-1$, denote $D :=nq(2(n-1)-nq)-2(n-1)$. Taking $x=\frac{nq(n-2)-(n-1)}{D} \in [\frac{1}{2}, 1]$ and $y = \frac{nq(2n-5)-2(n-2)}{D} \in [1, 2]$, we obtain the bound $g_n(\sigma) \leq \max(1-\frac{1}{2n} + \frac{x}{n} -\frac{qy}{2}, \frac{2}{n} + (1 - \frac{1}{n} - q)y, 1 - \frac{1}{n} + (\frac{2}{n}-q)x) = \frac{(2n-3)q(n+1-nq)-2(n-1)}{D}$.

    If $1-\frac{1}{2n} \leq \sigma \leq 1$, so that $n-1 \leq nq \leq n$, denote $\Delta :=1-n(1-q)^2$. Taking $x= \frac{nq-n+2}{\Delta} \in [\frac{n}{n-1}, 2]$ and $y = \frac{(n+2)q-n+1}{\Delta} \in [2, 3]$, we obtain in this final case the bound $g_n(\sigma) \leq \max(\frac{1}{2}+\frac{1}{n}+\frac{(n-1)x}{2n}-\frac{qy}{2}, (1-q)y, \frac{2}{n}+(1-\frac{1}{n}-q)x) = \frac{(1-q)((n+2)q-n+1)}{\Delta}$. 

    Combining the bounds for each range gives the claimed result.
\end{proof}

\begin{remark}
    The exponent $1+\frac{d(1-\sigma)}{3-2\sigma}$ of $T$ is chosen for simplicity, and is not optimal.
\end{remark}


\bibliographystyle{alpha}
\bibliography{references}

\end{document}